\documentclass[a4paper,12pt, reqno]{amsart}
\usepackage{amssymb,amsmath,latexsym}
\usepackage{mathrsfs}
\usepackage[shortlabels]{enumitem}
\usepackage{color}
\usepackage{comment}
\usepackage{xcolor}
\usepackage{stackengine}
\usepackage{calc}
\usepackage{physics}
\usepackage{accents}
\allowdisplaybreaks

\newcommand{\dbtilde}[1]{\accentset{\approx}{#1}}

\newcommand{\wt}{\widetilde}

\newcommand{\mb}{\mathbb}
\newcommand{\mc}{\mathcal}

\newtheorem{theorem}{Theorem}[section]
\newtheorem{definition}[theorem]{Definition}
\newtheorem{remark}[theorem]{Remark}
\newtheorem{lemma}[theorem]{Lemma}
\newtheorem{prop}[theorem]{Proposition}
\newtheorem{example}[theorem]{Example}
\newtheorem{cor}[theorem]{Corollary}
\title{stability property for the quantum jump operators of an open system}
\author{Marius Junge and Peixue Wu}
\begin{document}
\begin{abstract}
We show the continuity property of spectral gaps and complete Logarithmic constants in terms of the jump operators of Lindblad generators in finite dimensional setting. Our method is based on the bimodule structure of the derivation space and the technique developed in \cite{Paulsen09}. Using the same trick, we also show the continuity of the $g^2(0)$ constant used to distinguish quantum and classical lights in quantum optics.
\end{abstract}
\maketitle
\section{Introduction}
In quantum information theory, open quantum systems are described by Lindblad generators, the generators of a semigroup of completely positive unital (or trace preserving) maps.

In this paper we investigate the mathematical connection between the Lindblad generator $L: \mc B(\mc H)\to \mc B(\mc H)$ on finite dimensional space:
$$ L(x) = i[H,x]+\sum_j \big[ V_j^*xV_j -\frac{1}{2}(xV_j^*V_j+V_j^*V_jx)\big], V_j \in \mc B(\mc H), H = H^* \in \mc B(\mc H), $$
and its corresponding \emph{jump map} which is given by
 \[ \Psi_{L}(x) = \sum_j V_j^*xV_j.  \]
This map is certainly  completely positive, but not unital in general. The jump map contains a lot of information of the physical system. For example, in quantum optics, the number \emph{photon emission rate}
 \[ R(\rho) = \Tr(\Psi(1)\rho) \]
corresponds to the probability of detecting exactly one photon at time $t=0$ spontaneously. The operators $V_j$ are called \emph{jump operators}, so that
$R_j(\rho)=\Tr(a_j^*a_j\rho)$ is the probability of measuring a photon of type $j$.

We will show that the jump operators can be recovered from the Lindbladian itself. Then the problem of choosing the jump operators turns out to be equivalent to find the Kraus operators of a completely positive map. The main, standard tools in this analysis is the gradient form introduced in \cite{Lindblad76}:
 \[ 2\Gamma_{L}(x,y) = L(x^*y)-L(x^*)y-x^*L(y) \]
Note that the gradient form for the classical heat kernel in $\Gamma(f,g)=(\nabla f,\nabla g)$, called carr\'{e} du champs by P.A. Meyer \cite{Meyer}. For Lindblad operators, we find
 \[ \Gamma(x,y) = \sum_j [V_j,x]^*[V_j,y].\]
Note that $\delta_V(x)=[V,x]$ is the noncommutative version of a differential operator. This geometric point of view is well-accepted and facilitates introducing geometric tools to study the properties of quantum Markov semigroup, see \cite{CM17,CM20, GJL20,GJL21,GR22,LJL20,Wirth22, WirthZhang21} and references therein for more information.
Our main contribution is to build the analogy between Lindblad generators and jump maps in the following sense, which generalizes the work of Lindblad \cite{Lindblad76}
\[ \begin{array}{c|c}
  \text{Lindbladian}& \text{cp map} \\ \hline
  &  \\  \Gamma_{L}(\cdot, \cdot) =\sum_j \delta_{V_j} (\cdot)^*\delta_{V_j}(\cdot) & \Psi(\cdot) = \sum_j V_j^*(\cdot)V_j \\
  \Gamma_{L}\le_{cp} C \Gamma_{L'} & \Psi\le_{cp} C \Psi'\\
  \delta_{V_j}= \sum_k \alpha_{jk}\delta_{\wt V_k} &
  V_j= \sum_k \alpha_{jk}\wt V_k \\
  \end{array}  \]
The last two lines refer to different kind of comparisons. For a completely positive map a unitary in the environment does not change the map, whereas a linear transformation preserves the cp-order, i.e. $V_j=\sum_k \alpha_{jk}\wt V_k$ is equivalent to
 \[ (id_R \otimes \Psi)(x^*x)\le C (id_R\otimes \Psi')(x^*x) \]
for all matrices $x\in \mb M_d(\mb M_n),\dim R = d, \dim \mc H = n$. The same remains true in the category of derivation operators  on matrix algebras, provided we assume in addition that $V_j$ are matrices of trace $0$. Our main tool is based on the bimodule structure of the derivation space. See section 3-5 for details.

As applications, we show several stability properties based on the above diagram. For open systems with photon emissions, the quantity $g^{(2)}(0)$ reflects the probability of two photons being emitted almost simultaneously compared to the square probability of exactly one photon being emitter, and is given by the correlation function \cite{Optics2001}:
 \[ g^2(0) = g^2_L(0) = \frac{\Tr(\Psi_L\circ \Psi_L(I_n)\rho)}{\Tr(\Psi_L(I_n)\rho)^2}.\]
The notation $0$ in the quantity $g^2(0)$ comes from a limit procedure.  $g^2(0)>1$ implies superradiance, see \cite{Dicke54, AM22}. The interpretation is intuitive: $g^2(L)=1$ means that photon emission is stationary. $g^2(L)<1$ and $g^2(L)>1$ is sometimes referred to non-classical and classical light.

The following informal statement summerizes the main results discussed above:
\begin{theorem}(Informal) The map $L\to \Psi_{L}$ is injective and continuous. In particular, the photon emission rate $R$ and $g^{(2)}(0)$ depends continuously on $L$.
\end{theorem}
Another important application is the stability property of functional inequalities. In fact, our result implies a stability property of \emph{complete Logarithmic Sobolev inequality}(CMLSI), which is undergoing active research nowadays. We refer the reader to \cite{CM17, CM20, GJL20, GJL21, GR22, JLR19} for discussions on this topic. Our result is the following:
\begin{theorem}
Suppose $L: \mc B(\mc H) \to \mc B(\mc H)$ is a Lindblad generator with $\sigma-$detailed balance. Its fixed point algebra is $\mc N$. Then $\forall \varepsilon>0$, there exists $\delta = \delta(\varepsilon,L)> 0$, such that for any Lindblad generator $L'$ with $\sigma-$detailed balance which has fixed point algebra $\mc N$, we have
\begin{equation}
\|L-L'\|_{L_2 \to L_2} < \delta \implies
(1-\varepsilon)\text{CMLSI}(L)\le \text{CMLSI}(L').
\end{equation}
\end{theorem}
\noindent Similar argument can be applied to get stability property of \emph{spectral gap}.

Notational conventions:
\begin{itemize}
\item $\text{Tr}$ is the usual trace of matrix. $\tau: = \frac{1}{n}\text{Tr}$ is the normalized trace.
\item $\mc B(\mc H)$ is the set of bounded operators on the Hilbert space $\mc H$. 
\item The capital letter $I$ is denoted as identity superoperator: $I: \mc B(\mc H)\to \mc B(\mc H)$. $I_n$ is denoted as identity operators, the subscript $n$ denotes the dimension of the underlying Hilbert space.
\end{itemize}

\section{Preliminaries}
\subsection{Quantum Markov semigroup and gradient form}

First we recall the general definition of quantum Markov semigroup, which is a family of completely positive maps with certain conditions.
\begin{definition}
Suppose $\mathcal{M} \subseteq \mathcal{B}(\mathcal{H})$ is a subalgebra. We say that $T_t: \mathcal{M} \rightarrow \mathcal{M}$ is a (unital) quantum Markov semigroup if
\begin{itemize}
\item $\forall s,t \ge 0$, $T_sT_t = T_{s+t}, T_0 = I$. 
\item $\forall t \ge 0$, $T_t$ is normal, unital, completely positive.
\item $\forall x \in \mathcal{M}$, $T_tx \rightarrow x$ in ultra-weak topology as $t$ tends to 0. 
\end{itemize}
\end{definition}
In this paper, we will assume \textbf{$$\mathcal{M} = \mathcal{B}(\mc H), $$} where $\mathcal{H}$ is a Hilbert space of complex dimension $n\ge 2$. In other words, $\mathcal{M} = \mb M_n(\mathbb{C})$ for some $n\ge 2$. This will exclude some examples but some of our results are still true in more general settings.
Denote 
\begin{itemize}
\item $\mathcal{I}_n = \{(i,j)| 1\le i,j \le n\}$
\end{itemize} 
The Hilbert-Schimidt inner product $\langle\cdot,\cdot \rangle_{\text{HS}}$ on $\mathcal{M}$ is defined as \begin{equation}
\langle x,y \rangle_{\text{HS}} : = \tau(x^*y)
\end{equation}
Then $\mathcal{M}$ is a Hilbert space with respect to Hilbert-Schimidt inner product, denoted as $L^2(\mathcal{M},\tau)$. Since $\mathcal{M}$ is finite dimensional, we actually have norm continuity: 
\begin{equation}\label{norm continuous}
\lim_{t \rightarrow 0}||T_t -I|| = 0.
\end{equation}
Under \eqref{norm continuous}, we have a bounded normal linear operator $L: \mathcal{M} \rightarrow \mathcal{M}$, which is called the generator of quantum Markov semigroup, such that $T_t = e^{tL}$. It is easy to see that $L$ is *-preserving: $L(x^*) = L(x)^*,\forall x \in \mathcal{M}$ and $L(I) = 0$. The following theorem characterizes the generator of quantum Markov semigroup satisfying \eqref{norm continuous}, which is a combination of the main theorems of \cite{Alicki76, GKS75, Lindblad76}.

\begin{theorem} \label{Main theorem}
Suppose $L: \mathcal{M} \rightarrow \mathcal{M}$ is a bounded normal linear operator such that $L$ is *-preserving and $L(I) = 0$. The following are equivalent: 
\begin{enumerate}
\item $\{T_t = e^{tL}\}_{t \ge0}$ is a quantum Markov semigroup satisfying \eqref{norm continuous}.
\item $L$ has the form 
\begin{equation}\label{general form}
Lx = \sum_{i=1}^m \big[2V_i^*xV_i - (V_i^*V_ix + x V_i^*V_i)\big] + i[H,x]
\end{equation}
where $H,V_i \in \mathcal{M}$ and $H$ is self-adjoint and $m \ge 0$.
\item For any orthonormal basis $\{F_{\alpha}\}_{\alpha \in \mathcal{I}_n}$ of $L^2(\mathcal{M},\tau)$, where $\mathcal{I}_n = \{(i,j)| 1\le i,j \le n\}$, with $F_{(1,1)}= I_n$, $L$ can be written as  
\begin{equation} \label{GKS form}
Lx = i[H,x] + \sum_{\alpha,\beta \in \mathcal{I}_n \backslash \{(1,1)\}} c_{\alpha,\beta} \big[2F_{\alpha}^*xF_{\beta} - F_{\alpha}^*F_{\beta}x - xF_{\alpha}^*F_{\beta}\big]
\end{equation}
where the $(n^2-1) \times (n^2-1)$ matrix $(c_{\alpha,\beta})_{\alpha,\beta \in \mathcal{I}_n \backslash \{(1,1)\}}$ is a positive matrix.
\item The gradient form of the generator is completely positive, the definition of which is given below.
\end{enumerate}
\end{theorem} 
\begin{proof}
The equivalence of \textit{(1),(2),(4)} is given by \cite{Lindblad76}. The equivalence of \textit{(1),(3)} is given by \cite{GKS75}.
\end{proof}

Now we present another important definition: Gradient form of a Lindblad generator which is a non-commutative analogue of \textit{carr\'e du champ}:
\begin{definition}
For any normal bounded linear map $L: \mathcal{M} \rightarrow \mathcal{M}$, the gradient form of $L$ is given as $\Gamma_L: \mathcal{M} \times \mathcal{M} \rightarrow \mathcal{M}$:
\begin{equation}
\Gamma_L(x,y) = L(x^*y) - x^*L(y) - L(x^*)y,\ \forall x,y \in \mathcal{M}.
\end{equation}
\end{definition}
\noindent The gradient form $\Gamma_L$ is called \textit{completely positive} if
\begin{equation}
\forall m\ge 1, \forall X_m \in \mb M_m(\mathcal{M}) = \mb M_m \otimes \mathcal{M},\ \Gamma_{id_{\mb M_m} \otimes L}(X_m,X_m) \ge 0
\end{equation}
where 
$$\begin{aligned}
id_{\mb M_m} \otimes L: & \mb M_m \otimes \mathcal{M} \rightarrow \mb M_m \otimes \mathcal{M} \\
& x \otimes y \mapsto x \otimes L(y)
\end{aligned}$$
and $\Gamma_{id_{\mb M_m} \otimes L}: \mb M_m(\mathcal{M}) \times \mb M_m(\mathcal{M}) \rightarrow \mb M_m(\mathcal{M})$ is given by: $\forall X_m,Y_m \in \mb M_m(\mathcal{M})$,
$$
\begin{aligned}
&\Gamma_{id_{\mb M_m} \otimes L}(X_m,Y_m) = (id_{\mb M_m} \otimes L)(X_m^*Y_m) - X_m^*(id_{\mb M_m} \otimes L)(Y_m) - (id_{\mb M_m} \otimes L)(X_m^*)Y_m.
\end{aligned}$$
Similar to the idea of Choi matrix of a quantum channel, the complete positivity of gradient form is equivalent to the positivity of the matrix associated to the gradient form:

\begin{prop} \label{comparison}
For any bounded linear map $L: \mb M_n \to \mb M_n$, the following are equivalent
\begin{enumerate}
\item The gradient form $\Gamma_L$ is completely positive.
\item For any $m\ge 1$, and $\{x_i\}_{1\le i\le m} \subset \mb M_n$, $(\Gamma_L(x_i,x_j))_{1\le i,j\le m}$ is a positive matrix.
\item \begin{equation} \label{big matrix}
m_{L}: = \bigg[\Gamma_L(e_{rs},e_{tv})\bigg]_{(r,s),(t,v)} = \sum_{1\le r,s,t,v \le n} \ket{rs}\bra{tv} \otimes \Gamma_L(e_{rs},e_{tv})
\end{equation}
is a positive matrix.
\end{enumerate}
\end{prop}
\begin{proof}
First the equivalence of $(1) \iff (2)$ follows from direct calculation. For $(2) \implies (3)$, we choose $m = n^2$ many matrices $\{e_{rs}\}_{1\le r,s \le n}$ and we can get $m_L$ is positive directly.  
Finally, we need to show $(3) \implies (2)$. For any $m\ge 1$, and $\{x_i\}_{1\le i\le m} \subset \mb M_n$, we denote $$x_i = \sum_{1\le r,s\le n} \rho_{rs}^i e_{rs}.$$
Then by the bilinear property of $\Gamma_L$, we have 
\begin{equation} \label{equiv step}
\Gamma_L(x_i,x_j) = \sum_{1\le r,s,t,v \le n}\overline{\rho_{rs}^i} \rho_{tv}^j \Gamma(e_{rs}, e_{tv}), \rho_{rs}^i \in \mb C.
\end{equation}
To show that $(\Gamma_L(x_i,x_j))_{i,j}$ is positive, it is equivalent to show that 
for any $y_i \in \mb M_n, 1\le i \le m$, $$\sum_{1\le i,j\le m} y_i^* \Gamma_L(x_i,x_j)y_j$$
is positive as an $n$ by $n$ matrix. We plug \eqref{equiv step} into the above matrix, we have 
\begin{align*}
\sum_{1\le i,j\le m} y_i^* \Gamma_L(x_i,x_j)y_j = \sum_{1\le r,s,t,v \le n} \sum_{1\le i,j \le m}y_i^* \overline{\rho_{rs}^i} \rho_{tv}^j \Gamma_L(e_{rs}, e_{tv}) y_j.
\end{align*}
If we denote $X_{rs} = \sum_{1\le i \le m} \rho_{rs}^i y_i \in \mb M_n$, the above is equal to 
$$\sum_{1\le r,s,t,v \le n} X_{rs}^* \Gamma_L(e_{rs},e_{tv})X_{tv},$$
which is positive by the definition of $m_L$.
\end{proof}
Now recall the definition of the order of the gradient form, first introduced in \cite{Lindblad76}: 
\begin{definition}
For two generators $L,L'$ of the form \eqref{general form},
\begin{equation}\label{definition: grad order}
\Gamma_L \le_{cp} \Gamma_{L'} \iff \forall m\ge 1, \forall X_m \in \mb M_m(\mc M), \Gamma_{id_{\mb M_m} \otimes L}(X_m,X_m) \le \Gamma_{id_{\mb M_m} \otimes L'}(X_m,X_m) 
\end{equation}
\end{definition}
By Proposition \ref{comparison}, we have 
$$\Gamma_{\mc L} \le_{cp} C\Gamma_{\mc L'} \iff m_{\mc L} \le Cm_{\mc L'}, C >0.$$

Now we introduce a notion called \textit{standard form} of a Lindblad generator. Let us see how to express a Lindblad generator with the form of \textit{(2)} if we know $L$ is given by the form of \textit{(3)}: 
$$Lx = i[H,x] + \sum_{\alpha,\beta \in \mathcal{I}_n \backslash \{(1,1)\}} c_{\alpha,\beta} \big[2F_{\alpha}^*xF_{\beta} - F_{\alpha}^*F_{\beta}x - xF_{\alpha}^*F_{\beta}\big]$$
We choose a $(n^2-1) \times (n^2-1)$ unitary matrix $U$ such that it diagonalizes the positive matrix $(c_{\alpha,\beta})_{\alpha,\beta \in \mathcal{I}_n \backslash \{(1,1)\}}$, i.e., 
\begin{equation}
c_{\alpha,\beta} = \sum_{\gamma \in \mathcal{I}_n \backslash \{(1,1)\}} (U^*)_{\alpha,\gamma} c_{\gamma} U_{\gamma,\beta} = \sum_{\gamma \in \mathcal{I}_n \backslash \{(1,1)\}} \overline{U_{\gamma,\alpha}} c_{\gamma} U_{\gamma,\beta}, \forall \alpha,\beta \in \mathcal{I}_n \backslash \{(1,1)\}
\end{equation}
where $c_{\gamma} \ge 0$ because $(c_{\alpha,\beta})_{\alpha,\beta \in \mathcal{I}_n \backslash \{(1,1)\}}$ is a positive matrix. Then if we define for $\gamma \in \mathcal{I}_n \backslash \{(1,1)\}$,
\begin{equation}
V_{\gamma} = \sum_{\alpha \in \mathcal{I}_n \backslash \{(1,1)\}}U_{\gamma,\alpha}F_{\alpha}
\end{equation}
we have 
\begin{equation}
Lx = i[H,x] + \sum_{\gamma \in \mathcal{I}_n \backslash \{(1,1)\}}c_{\gamma}\big[2V_{\gamma}^* x V_{\gamma} - V_{\gamma}^*V_{\gamma}x - x V_{\gamma}^*V_{\gamma}\big]
\end{equation}
Moreover, $\big \{I_n, V_{\gamma}; \gamma \in \mathcal{I}_n \backslash \{(1,1)\} \big\}$ is an orthonormal basis of $L^2(\mathcal{M},\tau)$. This leads to the following definition:
\begin{definition} 
The standard form of a generator $L$ of a quantum Markov semigroup is given by 
\begin{equation}\label{standard form}
Lx = i[H,x] + \sum_{j = 1}^m \big[V_j^* x V_j - \frac{1}{2} (V_j^*V_jx + x V_j^*V_j)\big]
\end{equation}
where $\big \{I_n, V_j; 1 \le j \le m \big\}$ is an orthonormal set of $L^2(\mathcal{M},\tau)$.
\end{definition}
\begin{remark}
By the definition of Hilbert Schmidt inner product, $\tau(V_j) = 0$ for all $1 \le j \le m$ where $\{V_j\}$ is given by \eqref{standard form}.
\end{remark}

\subsection{Quantum Markov semigroups with $\sigma-$detailed balance condition}
For any linear operator $T: \mathcal{M} \rightarrow \mathcal{M}$, we define its adjoint $T_*$ by
\begin{equation}
\langle T(x),y \rangle_{\text{HS}} = \langle x,T_*(y) \rangle_{\text{HS}}
\end{equation} 
We have a general class of operators with symmetry assumptions, known as the operator with $\sigma-$detailed balance condition. Denote the density matrix as 
\begin{itemize}
\item $\mathcal{D}_{\mathcal{H}}:=\{\rho \in \mathcal{M}: \rho \ge 0, \text{Tr}(\rho)=1\}$
\item $\mathcal{D}_{\mathcal{H}}^+:=\{\rho \in \mathcal{M}: \rho > 0, \text{Tr}(\rho)=1\}$
\end{itemize}
Given $\sigma \in \mathcal{D}_{\mathcal{H}}^+$, we can define a family of inner products: for $s\in [0,1]$
\begin{equation}
\langle x,y \rangle_{s} = \text{Tr}(\sigma^s x^* \sigma^{1-s} y)
\end{equation}
More generally, for any measurable function $f: (0,+\infty) \rightarrow (0,+\infty)$, we can define 
\begin{equation}
\langle x,y \rangle_{f} = \text{Tr}(\sigma x^* f(\sigma)yf(\sigma)^{-1})
\end{equation}
We have the following special classes of inner product:
\begin{itemize}
\item GNS-inner product: $\langle x,y \rangle_1 = \text{Tr}(\sigma x^*y)$.
\item KMS-inner product: $\langle x,y \rangle_{\frac{1}{2}} = \text{Tr}(\sigma^{\frac{1}{2}} x^* \sigma^{\frac{1}{2}} y)$.
\item Bogoliubov-Kubo-Mori(BKM)-inner product: $\langle x,y \rangle_{f_0} = \int_0^1 \text{Tr}(\sigma^{1-s} x^* \sigma^{s} y)ds$, where $f_0(t) = \frac{t-1}{\log t} = \int_0^1 t^sds$
\end{itemize}
\begin{definition}(operator with $\sigma-$detailed balance)
\\
We say a linear operator $T: \mathcal{M} \rightarrow \mathcal{M}$ satisfies $\sigma-$detailed balance condition if it is self-adjoint with respect to GNS-inner product, i.e., 
$$\langle Tx,y \rangle_1 = \langle x,Ty \rangle_1.$$
We say a linear operator $T: \mathcal{M} \rightarrow \mathcal{M}$ is self-adjoint if $T = T_*$.
\end{definition}
\begin{remark}
It is shown in \cite{CM17} that if $T$ is linear, $*-$preserving which satisfies $\sigma-$detailed balance condition, then $T$ is self-adjoint with respect to $\langle \cdot,\cdot \rangle_{f}$ for any measurable function $f$. In other words, operators with $\sigma-$detailed balance(GNS-symmetric operators) form the smallest class among $\text{GNS}, \text{KMS}, \text{BKM}$-symmetric operators.
\end{remark}
\begin{remark}
If we take $\sigma = \frac{1}{n}I_n$, then $T$ satisfies $\sigma-$detailed balance condition means $T$ is a self-adjoint operator. In this sense $\sigma-$detailed balance condition is more general than self-adjoint.
\end{remark}
We say a quantum Markov semigroup $\{T_t\}_{t\ge0}$ satisfies $\sigma-$detailed balance condition for $\sigma \in \mathcal{D}_{\mathcal{H}}^+$, if for any $t\ge 0$, $T_t$ satisfies $\sigma-$detailed balance condition. In terms of the standard form \eqref{standard form}, the jump operators $\{V_j\}_{j \in \mc J}$ have some special properties. In fact they appear in pairs: if $V_j$ appears as a jump operator, then $V_j^*$ also appears as a jump operator up to a multiplicative constant. The details are given as follows:
\begin{theorem}(\cite{Alicki76,GKFV77})
Suppose the quantum Markov semigroup $\{T_t\}_{t\ge0}$ satisfies $\sigma-$detailed balance condition for $\sigma \in \mathcal{D}_{\mathcal{H}}^+$. Denote $\{\sigma_1,\sigma_2,\cdots,\sigma_n\} \subseteq (0,+\infty)$ as the eigenvalues of $\sigma$. Then the standard form of the generator $L$ of $T_t = e^{tL}$ is given as 
\begin{equation}\label{generator with detailed balance}
Lx =\sum_{\gamma \in \mathcal{I}_n \backslash \{(1,1)\}} c_{\gamma}[V_{\gamma}^*xV_{\gamma} - \frac{1}{2}(xV_{\gamma}^*V_{\gamma} + V_{\gamma}^*V_{\gamma}x)],
\end{equation}
where $\{I_n, V_{\gamma};\gamma \in \mathcal{I}_n \backslash \{(1,1)\}\}$ constitutes the eigen-basis(in particular, it is an orthonormal basis) of the modular operator 
\begin{equation}
\begin{aligned}
\Delta_{\sigma}: & L^2(\mathcal{M},\tau) \rightarrow L^2(\mathcal{M},\tau) \\
& \rho \mapsto \sigma \rho \sigma^{-1}.
\end{aligned}
\end{equation}
Moreover, the constants $\{c_{\gamma}; \gamma \in \mc I_n \backslash \{(1,1)\}\}$, the eigenvalues of $\sigma$ and the eigen-basis of $\Delta_{\sigma}$ satisfy the following relations: for any $\gamma = (i,j) \in \mathcal{I}_n \backslash \{(1,1)\}$, denote $\gamma' = (j,i)$, we have
\begin{itemize}
\item $\Delta_{\sigma}(V_{\gamma}) = \sigma V_{ij}\sigma^{-1} = \frac{\sigma_i}{\sigma_j}V_{ij}$.
\item $V_{\gamma} = V_{\gamma}^*,\ \text{if}\ \sigma_i = \sigma_j;\ V_{\gamma'}  = V_{\gamma}^*,\ \text{if}\ \sigma_i \neq \sigma_j$.
\item $c_{\gamma}\ge0$ satisfies $c_{\gamma}\sigma_j = c_{\gamma'}\sigma_i\ (c_{ij}\sigma_j = c_{ji}\sigma_i)$.
\end{itemize}
\end{theorem}
\begin{remark}
For readers used to the notation in \cite{CM17}, for any $\gamma = (i,j)$, if we define $$e^{-\omega_{\gamma}} = \frac{\sigma_i}{\sigma_j},$$ 
which are the eigenvalues of $\Delta_{\sigma}$, then using the relation 
\begin{align*}
& V_{\gamma} = V_{\gamma}^*,\ \text{if}\ \sigma_i = \sigma_j;\ V_{\gamma'}  = V_{\gamma}^*,\ \text{if}\ \sigma_i \neq \sigma_j, \\
& c_{\gamma}\sigma_j = c_{\gamma'}\sigma_i\ (c_{ij}\sigma_j = c_{ji}\sigma_i)
\end{align*}
we can rewrite the generator with the same form as in \cite{CM17}: 
\begin{align*}
Lx & = \sum_{\gamma \in \mathcal{I}_n \backslash \{(1,1)\}} e^{-\omega_{\gamma}/2}\big(W_{\gamma}^*xW_{\gamma} - \frac{1}{2}(xW_{\gamma}^*W_{\gamma} + W_{\gamma}^*W_{\gamma}x)\big) \\
& = \sum_{\gamma \in \mathcal{I}_n \backslash \{(1,1)\}} e^{-\omega_{\gamma}/2}\big( W_{\gamma}^*[x,W_{\gamma}] + [W_{\gamma}^*,x]W_{\gamma}\big) \\
& = \sum_{\gamma \in \mathcal{I}_n \backslash \{(1,1)\}} \big( e^{-\omega_{\gamma}/2}W_{\gamma}^*[x,W_{\gamma}] +  e^{\omega_{\gamma}/2}[W_{\gamma},x]W_{\gamma}^*\big),
\end{align*}
where $\{I_n,W_{\gamma};\gamma \in \mc I_n \backslash \{(1,1)\}\}$ are orthonormal eigenvectors(not necessarily basis) of $\Delta_{\sigma}$ in $L_2(\mc M,\tau)$.
\end{remark}
Now we rewrite the expression \eqref{generator with detailed balance} by classifying the different eigenvalues $e^{-\omega_{\gamma}}$. Suppose 
$$Lx = \sum_{j \in \mc J}e^{-\omega_j/2} \big[ V_j^*xV_j -\frac{1}{2}(V_j^*V_jx + xV_j^*V_j) \big].$$
We partition the index set $\mc J$ as $\{\mc J_1, \cdots, \mc J_k, \cdots \mc J_l\}$ such that $|\mc J|=m = \sum_{k=1}^l |\mc J_k|, |\mc J_k| = m_k$, and $\forall k \le l, \forall j \in \mc J_k$, we have $ \omega_j = \omega_k.$ This means we only have $l$ different eigenvalues appearing in the generator $L$. Then
\begin{equation}\label{generator partition}
Lx = \sum_{k=1}^l e^{-\omega_k/2} \sum_{j \in \mc J_k}\big[ V_j^*xV_j -\frac{1}{2}(V_j^*V_jx + xV_j^*V_j) \big]. 
\end{equation}
We will use this notation in later sections.

Another simple implication is that for self-adjoint generator, our $\sigma$ is given by $\frac{1}{n}I_n$, which means the eigenvalues of $\sigma$ are identical to $\frac{1}{n}$, thus all $V_{\gamma}$ are self-adjoint. This gives the well-known form of self-adjoint generator:
\begin{cor}
Suppose the quantum Markov semigroup $\{T_t\}_{t\ge0}$ is self-adjoint, then the generator $L$ has the following standard form 
\begin{equation}
Lx = \sum_{j \in \mathcal{J}} \big[V_jxV_j - \frac{1}{2}(V_j^2x + xV_j^2)\big],
\end{equation}
where $V_j$ are self-adjoint matrices for any $j \in \mc J$. Moreover, $\{I_n, V_j;j \in \mathcal{J}\}$ are orthonormal sets and $\mc J$ is an index set with finite elements: $|\mathcal{J}| \le n^2 - 1$.
\end{cor}

\section{Comparison principle for general quantum Markov semigroups}
In this section, we will prove the following criteria for comparing two gradient forms given by general Lindblad generators:
\begin{theorem}\label{main:comparison}
Suppose $L,L'$ are two generators of the form \eqref{general form}, i.e., 
$$
\begin{aligned}
& Lx = i[H,x] + \sum_{i=1}^m V_i^*xV_i - \frac{1}{2}(V_i^*V_i x + x V_i^*V_i),\\
& L'x = i[H',x] + \sum_{i=1}^{m'}V_i'^*xV_i' - \frac{1}{2}(V_i'^*V_i' x + x V_i'^*V_i')\end{aligned}$$
where $\mathcal{X} = \{V_i\}_{i=1}^m,\mathcal{X'}= \{V_i'\}_{i=1}^{m'}$ may not be orthonormal sets. Then $\Gamma_L \le C \Gamma_{L'}$ for some $C>0$ if and only if $\text{span}\{I_n, \mathcal{X}\} \subseteq \text{span}\{I_n, \mathcal{X'}\}$. Moreover, if we express our Lindblad generator as the standard form defined in \eqref{standard form}, then $\Gamma_L \le C \Gamma_{L'}$ for some $C>0$ if and only if $\text{span}\{\mathcal{X}\} \subseteq \text{span}\{\mathcal{X'}\}$.
\end{theorem}

To prove the theorem, we also need some results in operator algebra theory. Given $m \ge 1$, define the $\mc M-$bimodule $$C_m(\mc M):= \bigg\{\begin{pmatrix}
a_1 \\
a_2 \\
\cdots \\
a_m
\end{pmatrix} : a_i \in \mc M \bigg\}
$$ with $\mc M-$valued inner product $$\langle\begin{pmatrix}
a_1 \\
a_2 \\
\cdots \\
a_m
\end{pmatrix}, \begin{pmatrix}
b_1 \\
b_2 \\
\cdots \\
b_m
\end{pmatrix}\rangle = \sum_{k=1}^m a_k^*b_k \in \mc M.$$ 
Now we can define the derivation $ \delta_L$
\begin{equation} \label{defn: derivation}
\begin{aligned}
\delta_L: & \mc M \rightarrow C_m(\mc M) \\
& x \mapsto \begin{pmatrix}
[V_1,x] \\
[V_2,x] \\
\cdots \\
[V_m,x]
\end{pmatrix}
\end{aligned}
\end{equation}
By direct calculation, we see that $$\Gamma_L(x,y) = \sum_{k=1}^m [V_k,x]^*[V_k,y] = \langle\delta_L(x), \delta_L(y)\rangle$$
Now define an $\mc M-$right submodule determined by $L$:
\begin{equation}\label{defn:bimodule}
\Omega_{\delta_L}:= \{\sum_j \delta_L(x_j)y_j:\ x_j,y_j \in \mc M\} \subseteq C_m(\mc M)
\end{equation}
with inner product defined by 
$$\langle \sum_j \delta_L(x_j)y_j, \sum_k \delta_L(x'_k)y'_k\rangle_{\Omega_{\delta_L}}:= \sum_{j,k}y_j^* \Gamma_L(x_j,x'_k)y'_k$$
Using the chain rule $\delta(xy)= \delta(x)y + x\delta(y), x,y \in \mc M$, $\Omega_{\delta_L}$ is actually an $\mc M$-bimodule, given by 
$$\Omega_{\delta_L}:= \{\sum_j z_j\delta_L(x_j)y_j:\ x_j,y_j,z_j \in \mc M \} \subseteq C_m(\mc M)$$
By replacing $m$ by $m'$, we can similarly define $\Omega_{\delta_{L'}} \subseteq C_{m'}(\mc M)$. 
Define the map 
\begin{equation}\label{defn:module}
\begin{aligned}
& T: \Omega_{\delta_{L'}} \rightarrow \Omega_{\delta_{L}} \\
& \sum_j \delta_{L'}(x_j)y_j \mapsto \sum_j \delta_{L}(x_j)y_j
\end{aligned}
\end{equation}
By using the fact $\delta(xy)= \delta(x)y + x\delta(y)$ and the definition, it is easy to see that $T$ is actually a bimodule map.
\begin{lemma}
Suppose $C>0$ is a fixed constant.
\begin{equation} \label{module map}
\begin{aligned}
\Gamma_L \le C\Gamma_{L'} \iff \langle T(\xi),T(\xi)\rangle_{\Omega_{\delta_L}} \le C\langle \xi,\xi\rangle_{\Omega_{\delta_{L'}}},\ \forall \xi \in \Omega_{\delta_{L'}}
\end{aligned}
\end{equation}
\end{lemma}
\begin{proof}
For any $\xi \in \Omega_{\delta_{L'}}$, there exist $x_j,y_j \in \mc M, 1\le j \le l$, such that $$\xi = \sum_{j=1}^l \delta_{L'}(x_j)y_j. $$
Then $T(\xi) = \sum_{j=1}^l \delta_{L}(x_j)y_j$. If $\Gamma_L \le C\Gamma_{L'}$, using Proposition \ref{comparison}, we know that
$$\big(\Gamma_L(x_i,x_j)\big)_{1\le i,j\le l} \le C\big(\Gamma_{L'}(x_i,x_j)\big)_{1\le i,j\le l},$$
which implies $$\sum_{i,j}y_i^*\Gamma_L(x_i,x_j)y_j \le C\sum_{i,j}y_i^*\Gamma_{L'}(x_i,x_j)y_j.$$
The necessary part follows easily if we recall that 
$$\langle \sum_j \delta_L(x_j)y_j, \sum_j \delta_L(x_j)y_j\rangle_{\Omega_{\delta_L}}= \sum_{i,j}y_i^* \Gamma_L(x_i,x_j)y_j.$$
For the sufficiency part, note that if for any finitely many matrices $x_j,y_j \in \mc M$,
$\sum_{i,j}y_i^*\Gamma_L(x_i,x_j)y_j \le C\sum_{i,j}y_i^*\Gamma_{L'}(x_i,x_j)y_j$ implies
$$\big(\Gamma_L(x_i,x_j)\big)_{1\le i,j\le l} \le C\big(\Gamma_{L'}(x_i,x_j)\big)_{1\le i,j\le l}.$$
Applying Proposition \ref{comparison} again, we have $\Gamma_L \le C\Gamma_{L'}$.
\end{proof}

We can naturally define a norm on the $\mc M-$bimodule $C_m(\mc M)$: $\|\xi\| = \|\langle \xi,\xi\rangle\|$. Now suppose $\Gamma_L \le C\Gamma_{L'}$, from the above lemma we know the operator norm of $T$ satisfies $||T|| \le \sqrt{C}$. We have the following extension lemma:
\begin{lemma}
There exists a bimodule map $\widehat{T}: C_{m'}(\mc M) \rightarrow C_m(\mc M)$, such that $$\widehat{T}|_{\Omega_{\delta'}} = T.$$
\end{lemma}
\begin{proof}
Since $T: \Omega_{\delta_{L'}} \rightarrow C_m(\mc M)$ is a bounded bimodule map, we can use \cite[Proposition 8.2.2]{BM04} to get $||T|| = ||T||_{cb}$. By \cite{Ruan89}, $C_m(\mc M)$ is injective, which means for any operator space $W \subseteq Z$, and any completely bounded map $T:W \rightarrow C_m(\mc M)$, there exists a linear map $\widetilde{T}: Z \rightarrow C_m(M)$ extending $T$, such that $||T||_{cb} = ||\widetilde{T}||_{cb}$. Thus we can get an extension $\widetilde{T}: C_{m'}(\mc M) \rightarrow C_m(\mc M)$, such that $||\widetilde{T}||_{cb} = ||T||_{cb}$. Now take $\mu$ to be the normalized Haar measure on the unitary group $U(n) \subseteq \mc M$, and define 
\begin{equation}
\begin{aligned}
\dbtilde{T}&: C_{m'}(\mc M) \rightarrow C_m(\mc M) \\
& \xi \mapsto \int_{U(n)}\widetilde{T}(\xi u)u^* \mu(du)
\end{aligned}
\end{equation}
We claim that $\dbtilde{T}\big| _{\Omega_{\delta'}} = T$ and $\dbtilde{T}$ is a right module map. Indeed, for the first one, we have
$$\dbtilde{T}(x) = \int_{U(n)}\widetilde{T}(x u)u^* \mu(du) =\int_{U(n)}T(x u)u^* \mu(du) = \int_{U(n)}T(x) \mu(du) = T(x), \forall x \in \Omega_{\delta'}$$
To show $\dbtilde{T}$ is a right module map, we only need to show for any $v \in U(n)$, $\dbtilde{T}(\xi v) = \dbtilde{T}(\xi) v$ because any element in $\mc M = \mb M_n$ can be written as a linear combination of elements in $U(n)$. Indeed,
$$\dbtilde{T}(\xi v) = \int_{U(n)}\widetilde{T}(\xi vu)u^* \mu(du) =\int_{U(n)}\widetilde{T}(\xi vu)(vu)^*v \mu(du) = \int_{U(n)}\widetilde{T}(\xi u)u^*v \mu(du) = \dbtilde{T}(\xi)v $$
Note that for the third equality, we used left invariance of Haar measure. \\
Now we define $\widehat{T}$: 
\begin{equation}
\begin{aligned}
\widehat{T}&: C_{m'}(\mc M) \rightarrow C_m(\mc M) \\
& \xi \mapsto \int_{U(n)}u^*\dbtilde{T}(u\xi) \mu(du)
\end{aligned}
\end{equation}
We can check $\widehat{T}\big| _{\Omega_{\delta'}} = T$ and $\widehat{T}$ is a bimodule map, using similar argument as above. 
\end{proof}

\begin{lemma}
\begin{equation}
\begin{aligned}
& \{T: C_{m'}(\mc M)\rightarrow C_m(\mc M) \big|T\ bounded\ linear,T(axb) = aT(x)b,\forall x\in C_{m'}(\mc M), a,b \in \mc M \} \\
&= \mb M_{m\times m'} \otimes \mc M'
\end{aligned}
\end{equation}
where $\mc M'$ is the commutant of $\mc M$. In our case, since $\mc M' = \{cI_n\}$, the above is equal to $\mb M_{m\times m'} \otimes I_n$.
\end{lemma}
\begin{proof}
For any bounded right module map $T: C_{m'}(\mc M)\rightarrow C_m(\mc M)$, it is adjointable because $L^2(\mc M,\tau)$ is finite dimensional in this case(this means we can take its adjoint). Then it is well-known that \cite{Lance95} $$T \in \mb M_{m\times m'} \otimes \mathcal{L}(\mc M) \simeq  \mb M_{m\times m'}(\mc M),$$
where $\mc L(\mc M)$ is all the linear maps on $\mc M$. 

Suppose $T = \big[a_{ij}\big]_{1\le i \le m, 1\le j \le m'},\ a_{ij} \in \mc M$, seen as left multiplication. Then the left module property gives us for all $a \in \mc M$,
\begin{align*}
& T\bigg(a \begin{pmatrix}
x_1\\
\cdots\\
x_{m'}
\end{pmatrix}\bigg) = a T\bigg(\begin{pmatrix}
x_1\\
\cdots\\
x_{m'}
\end{pmatrix}\bigg) \\
& = \begin{pmatrix}
a_{11} & \cdots & a_{1m'} \\
\cdots &\cdots & \cdots \\
a_{m1} & \cdots & a_{mm'}
\end{pmatrix} 
\begin{pmatrix}
a x_1 \\ 
\cdots \\
a x_{m'}
\end{pmatrix} 
 = (I_m\otimes a)
\begin{pmatrix}
a_{11} & \cdots & a_{1m'} \\
\cdots &\cdots & \cdots \\
a_{m1} & \cdots & a_{mm'}
\end{pmatrix} 
\begin{pmatrix}
 x_1 \\ 
\cdots \\
 x_{m'}
\end{pmatrix} \\
& \iff \sum_{j=1}^{m'} a_{ij} a x_j = \sum_{j=1}^{m'} a a_{ij}x_j, \forall 1\le i \le m
\end{align*}
where $a,x_j \in \mc M,\ 1\le j \le m' $ are arbitrary, which implies $aa_{ij} = a_{ij}a$ thus $a_{ij} \in \mc M'$.
\end{proof}
\noindent \textbf{Proof of Theorem \ref{main:comparison}:} \\
Sufficiency: We define $D \in \mb M_{m \times n^2}(\mb M_n)$: 

\begin{equation}\label{D matrix:1}
D_L = \bigg([V_i,e_{rs}]\bigg)_{1\le i \le m, 1\le r,s \le n}.
\end{equation}
By direct calculation, we have 
\begin{align*}
D_L^*D_L = \big(\sum_{k=1}^m [V_k,e_{rs}]^*[V_k,e_{tv}]\big)_{rs,tv} =\big(\Gamma_L(e_{rs},e_{tv})\big)_{rs,tv} = m_L.
\end{align*}
Since $span\{V_j\}_{1\le j\le m} \subseteq span\{\wt V_j, I_n;1\le j \le m'\}$, there exists $\alpha_{ij}, c_i \in \mb C, 1\le i\le m, 1\le j \le m'$ such that $$V_i = \sum_{j=1}^{m'} \alpha_{ij}\widetilde{V_j} + c_i I_n$$ for any $1\le i \le m$. Then
\begin{equation} \label{D matrix:2}
\begin{aligned}
D_L & = \bigg([V_i,e_{rs}]\bigg)_{1\le i \le m, 1\le r,s \le n} = \bigg([\sum_{j=1}^{m'}\alpha_{ij}\wt V_j,e_{rs}]\bigg)_{1\le i \le m, 1\le r,s \le n} \\
& = \bigg(\sum_{j=1}^{m'}\alpha_{ij}[\wt V_j,e_{rs}]\bigg)_{1\le i \le m, 1\le r,s \le n} \\
& = (A\otimes I_n) D_{L'}
\end{aligned}
\end{equation}
where $A = (\alpha_{ij})_{1\le i\le m, 1\le j \le m'} \in \mb M_{m \times m'}$ and 
$D_{L'} = \bigg([\wt V_i,e_{rs}]\bigg)_{1\le i \le m', 1\le r,s \le n} \in \mb M_{m' \times n^2}(\mb M_n)$. \\
Thus we have 
\begin{equation}
m_L = D_L^*D_L = D_{L'}^*(A^*A \otimes I_n) D_{L'}\le C D_{L'}^*D_{L'} = Cm_{L'}.
\end{equation}
where $C = ||A^*A \otimes I_n||$, which is equivalent to $\Gamma_L \le C \Gamma_{L'}$.  

\noindent Necessity: Our assumption is $\Gamma_L \le C \Gamma_{L'}$. Recall we can define the bimodule map \begin{equation}
\begin{aligned}
& T: \Omega_{\delta_{L'}} \rightarrow \Omega_{\delta_{L}} \\
& \sum_j \delta_{L'}(x_j)y_j \mapsto \sum_j \delta_{L}(x_j)y_j
\end{aligned}
\end{equation}
By the above lemmas, we have for any $x \in \mc M$, 
\begin{align*}
T(\delta_{L'}(x)) = \widehat{T}|_{\Omega_{\delta'}}(\delta_{L'}(x)) 
& =\begin{pmatrix}
\alpha_{11}I_n & \alpha_{12}I_n & \cdots & \alpha_{1m'}I_n \\
\alpha_{21}I_n & \alpha_{22}I_n & \cdots & \alpha_{2m'}I_n \\
\cdots & \cdots & \cdots & \cdots\\
\alpha_{m1}I_n & \alpha_{m2}I_n & \cdots & \alpha_{mm'}I_n
\end{pmatrix}
\begin{pmatrix}
[\widetilde{V_1}, x] \\
[\widetilde{V_2}, x] \\
\cdots \\
[\widetilde{V_{m'}}, x]
\end{pmatrix}  \\
& = \begin{pmatrix}
[\sum_{j=1}^{m'}\alpha_{1j}\wt V_j, x] \\
[\sum_{j=1}^{m'}\alpha_{2j}\wt V_j, x] \\
\cdots \\
[\sum_{j=1}^{m'}\alpha_{mj}\wt V_j, x]
\end{pmatrix} = \begin{pmatrix}
[V_1, x] \\
[V_2, x] \\
\cdots \\
[V_m, x]
\end{pmatrix},
\end{align*}
which means for any $i = 1,2,\cdots,m$, 
\begin{equation}\label{linear span}
V_i = \sum_{j=1}^{m'} \alpha_{ij}\widetilde{V_j}+ c_i I_n
\end{equation}
i.e., $span \mc X \subseteq span\{\mc X', I_n\}$. In particular, if $\{V_i\}$ are trace zero orthonormal sets, we have $V_i = \sum_{j=1}^{m'} \alpha_{ij}\widetilde{V_j}$ thus $span \mc X \subseteq span \mc X'$.\qed

If we express $L,L'$ as the standard form. $\{V_i\}_{1\le i \le m}, \{\wt V_i\}_{1\le i \le m'}$ are the traceless orthonormal jump operators. If $\Gamma_L \le C\Gamma_{L'}$ for some constant $C>0$, then by our theorem, $span\{V_i\}_{1\le i \le m}\subseteq span\{\wt V_i\}_{1\le i \le m'}$, which implies $m\le m'$. If $m' > m$, we assume 
$$span\{V_i\}_{1\le i \le m}= span\{\wt V_i\}_{1\le i \le m}.$$
Then we can define the truncation $\widehat{L'}$ of $L'$ by 
\begin{equation}\label{truncation generator}
\widehat{L'}(x) = i[H',x] + \sum_{j=1}^m \big[\wt V_j^*x\wt V_j - \frac{1}{2}(\wt V_j^*\wt V_jx + x\wt V_j^*\wt V_j)\big].
\end{equation}
We can show that the redundant part of $L'$ does not affect the comparison of Lindblad generators: 
\begin{cor}\label{drop}
For any $C>0, C_1\ge 0$, 
\begin{align*}
\Gamma_L \le C\Gamma_{\widehat{L'}} + C_1 \Gamma_{L'-\widehat{L'}} \iff \Gamma_L \le C\Gamma_{\widehat{L'}}.
\end{align*}
\end{cor}
\begin{proof}
Sufficiency is obvious. We only need to prove the necessity part. If $\Gamma_L \le C\Gamma_{\widehat{L'}} + C_1 \Gamma_{L'-\widehat{L'}}$, then by Theorem \ref{main:comparison} and the assumption $span\{V_i\}_{1\le i \le m}= span\{\wt V_i\}_{1\le i \le m}$, we know that the bimodule map $T: \Omega_{\delta_{L'}}\to \Omega_{\delta_{L}}$ defined by \eqref{defn:module} is given by $A \otimes I_n$ as left multiplication, where $A= (\widehat{A}\ 0) \in \mb M_{m\times m'}$ and $\widehat{A} \in \mb M_m$. Then following the same procedure \eqref{D matrix:2} as in the proof of Theorem \ref{main:comparison}, we have $D_L = (A\otimes I_n)D_{L'}$ and $D_L = (\widehat{A}\otimes I_n)D_{\widehat{L'}}$.
\begin{align*}
D_L^*D_L =  D_{L'}^*(A^*A\otimes I_n)D_{L'} = D_{\widehat{L'}}^*(\widehat{A}^*\widehat{A}\otimes I_n)D_{\widehat{L'}}
\end{align*} 
Then we know that $\Gamma_L \le \|A^*A\| \Gamma_{\widehat{L'}}.$ Thus we only need to prove $$C \ge \|A^*A\|.$$
In fact, if $C < \|A^*A\| = \|\widehat{A}^*\widehat{A}\|$, we know that $C I_m - \widehat{A}^*\widehat{A}$ has at least one negative eigenvalue by the definition $\|A^*A\| = \lambda_{max}(A^*A)$. Then 
\begin{align*}
& C\widehat{L'} - L + C_1(L' - \widehat{L'})(x) \\
& = i[\wt H,x] + \sum_{1\le i,j\le m}(C\delta_{i,j} - \beta_{ij}) \big[ \wt V_i^*x\wt V_j - \frac{1}{2}(\wt V_i^*\wt V_jx + x\wt V_i^*\wt V_j) \big] \\
&+ \sum_{j=m+1}^{m'} \big[\wt V_j^*x\wt V_j - \frac{1}{2}(\wt V_j^*\wt V_jx + x\wt V_j^*\wt V_j)\big],
\end{align*}
where $(\beta_{ij}) = \widehat{A}^*\widehat{A}$. The coefficient matrix with respect to $\{\wt V_j\}_{1\le j\le m'}$ is given by $$\begin{pmatrix}
CI_m -\widehat{A}^*\widehat{A} & 0\\
0 & C_1 I_{m'-m}
\end{pmatrix}
$$
which is not a positive matrix. Thus by Theorem \ref{Main theorem}, $\Gamma_{C\widehat{L'} - L + C_1(L' - \widehat{L'})}$ is not completely positive, which contradicts to the assumption. 
\end{proof}

The proof of the above theorem also shows that the comparison of two Lindblad generators $L,L'$ is characterized by the comparison of their \textit{jump map} $\Psi_L,\Psi_{L'}$:
\begin{cor}\label{equivalent comparison}
Suppose $L,L'$ are two Lindblad generators, and $\Psi_L,\Psi_{L'}$ are their corresponding jump maps. Then for any constant $C>0$,
\begin{equation}
\Gamma_L \le C \Gamma_{L'} \iff \Psi_L \le C \Psi_{L'}
\end{equation}
\begin{proof}
Suppose $\Gamma_L \le C \Gamma_{L'}$. Then $span\{V_j\}_{1\le j \le m} \subseteq span\{\wt V_j\}_{1\le j \le m'}$ by Theorem \ref{main:comparison}, where $\{V_j\}_{1\le j \le m}, \{\wt V_j\}_{1\le j \le m'}$ are the traceless orthonormal jump operators of $L,L'$. Without loss of generality, we assume that (recall $m \le m'$)
$$span\{V_j\}_{1\le j \le m} = span\{\wt V_j\}_{1\le j \le m}.$$
Then there exists a full-rank matrix $A = (\alpha_{ij}) \in \mb M_m$, such that
\begin{equation}\label{linear span}
V_i = \sum_{j=1}^m \alpha_{ij} \wt V_j, \forall 1\le i\le m. 
\end{equation} 
Then we have 
\begin{align*}
m_L = D_L^*D_L = D_{\widehat{L'}}^*(A^*A\otimes I_n)D_{\widehat{L'}} \le \|A^*A\| D_{\widehat{L'}}^*D_{\widehat{L'}} \le \|A^*A\| m_{L'}.
\end{align*}
We can apply the same argument as in Corollary \ref{drop} to show that $C\ge \|A^*A\|$. 
On the other hand, recall that for two completely positive map, $\Psi_L \le C \Psi_{L'}$ if any only if their corresponding Choi matrices $\tau_{\Psi_L} \le C \tau_{\Psi_{L'}}$, where $\tau_{\Psi_L} = \big(\sum_{k=1}^m V_k^*\ket{r}\bra{s}V_k\big)_{1\le r,s \le n} \in \mb M_n(\mb M_n)$. Now define 
\begin{equation}
W_{\Psi_L} = \big(\bra{r}V_i\big)_{1\le i \le m, 1\le r\le n} \in \mb M_{m\times n}(\mb M_{1 \times n}).
\end{equation}
Then it is easy to check $\tau_{\Psi_L} = W_{\Psi_L}^*W_{\Psi_L}$. Moreover, by \eqref{linear span}, we have $ W_{\Psi_L} =  A W_{\Psi_{\widehat{L'}}}$. Thus 
\begin{align*}
\tau_{\Psi_L} = W_{\Psi_L}^*W_{\Psi_L} = W_{\Psi_{\widehat{L'}}}^*(A^*A) W_{\Psi_{\widehat{L'}}} \le \|A^*A\|\tau_{\Psi_{L'}} \le C \tau_{\Psi_{L'}}.
\end{align*}
For the other direction, classical result implies that $span\{V_j\}_{1\le j \le m} \subseteq span\{\wt V_j\}_{1\le j \le m'}$ and the constant $C \ge \|A^*A\|.$
\end{proof}
\end{cor}
\begin{remark}
A simple implication is that the gradient form $\Gamma_L$ uniquely determines the jump map $\Psi_L$. In fact, if $\Gamma_L = \Gamma_{L'}$, then from the above corollary $\Psi_L = \Psi_{L'}$.
\end{remark}

\section{The structure of $I-E_{\mathcal{N}}$}
In this section, we characterize the simple operator $-(I-E_{\mathcal{N}})$ as a Lindblad generator, where $E_{\mathcal{N}}$ is the conditional expectation onto $\mathcal{N}$ and $\mc N \subseteq \mc M = \mc B(\mc H)$ is a sub-algebra. In fact, since $E_{\mc N}$ is a unital completely positive map, its Kraus representation is given by 
$$E_{\mc N}(x) = \sum_{j}V_j^*xV_j, \sum_{j}V_j^*V_j = I_n.$$
Therefore, $-(I-E_{\mc N})$ is a Lindblad generator. It is well-known that if $\mc N = \bigoplus_{k=1}^l M_{n_k} \otimes 1_{r_k} \subseteq \mc B(\mc H)$, where the Hilbert space has the decomposition $\mc H = \bigoplus_{k=1}^{l} \mc H_k \otimes \mc K_k$, the conditional expectation $E_{\mc N}$ is given by 
\begin{equation}
E_{\mc N}(x) = \bigoplus_{k=1}^l \Tr_{\mc K_k}(P_k xP_k(1_{n_k}\otimes \tau_k))\otimes 1_{r_k},
\end{equation}
where $P_k$ is the projection of $\mc H$ onto $\mc H_k\otimes \mc K_k$ and $\tau_k$ is a family of density operators acting on $\mc K_k$ determining the conditional expectation. In particular, if $\tau_k = \frac{1}{dim \mc K_k}1_{r_k} = \frac{1}{r_k}1_{r_k}$, we denote our conditional expectation as $E_{\mc N,\tau}$, the \textit{trace preserving} conditional expectation.
We will show that as a Lindblad generator, the Kraus operators of $-(I-E_{\mc N,\tau})$ are maximal in some sense. The following simple case is a strong indication.
\begin{lemma}\label{I-E:primitive}
Suppose $\{I_n, V_j;j \in \mathcal{J}\}$ is an orthonormal basis of $L^2(\mathcal{M},\tau)$, then 
$$\sum_{i \in \mathcal{I}} \big[V_i^*xV_i - \frac{1}{2}(V_i^*V_ix + xV_i^*V_i) \big]= -n^2( I-E_{\tau}),$$
where $E_{\tau}(x) = \tau(x)I_n$.
\end{lemma}
\begin{proof}
Note that $\{\sqrt{n}e_{ij}\}$ is an orthonormal basis of $L^2(\mathcal{M},\tau)$, where $e_{ij} = \ket{i}\bra{j}$ is the matrix unit. We have by direct calculation,
\begin{align*}
\widetilde{L}x & := n \sum_{i,j} \big[e_{ij}^*xe_{ij} - \frac{1}{2}(xe_{ij}^*e_{ij} + e_{ij}^*e_{ij} x)\big] = -n^2(I-E_{\tau})(x).
\end{align*}
Now we denote $\{E_{\alpha}\}_{\alpha \in \mathcal{I}_n} = \{\sqrt{n} e_{ij}\}_{1\le i\le n,1\le j \le n},\ \{V_{\alpha}\}_{\alpha \in \mathcal{I}_n} = \{I_n,V_j;j\in \mathcal{J}\}$. Since they are both orthonormal basis of $L_2(\mc M,\tau)$, there exists a unitary $(U_{\alpha,\beta})_{\alpha,\beta \in \mathcal{I}_n}$, such that 
$$V_{\alpha} = \sum_{\beta \in \mathcal{I}_n}U_{\alpha,\beta}E_{\beta}$$
Thus 
\begin{align*}
& \sum_{\alpha \in \mathcal{I}_n} \big[V_{\alpha}^*xV_{\alpha} - \frac{1}{2}(V_{\alpha}^*V_{\alpha}x + x V_{\alpha}^*V_{\alpha})\big] \\
& = \sum_{\alpha,\beta,\gamma \in \mathcal{I}_n} \big[\overline{U_{\alpha,\beta}}U_{\alpha,\gamma}E_{\beta}^*xE_{\gamma} - \frac{1}{2}(\overline{U_{\alpha,\beta}}U_{\alpha,\gamma}E_{\beta}^*E_{\gamma}x + x \overline{U_{\alpha,\beta}}U_{\alpha,\gamma}E_{\beta}^*xE_{\gamma}\big] \\
& = \sum_{\beta, \gamma \in \mathcal{I}_n} \delta_{\beta,\gamma}\big[E_{\beta}^*xE_{\gamma} - \frac{1}{2}(xE_{\beta}^*E_{\gamma} + E_{\beta}^*E_{\gamma}x)\big] \\
& = n \sum_{i,j} \big[e_{ij}^*xe_{ij} - \frac{1}{2}(xe_{ij}^*e_{ij} + e_{ij}^*e_{ij} x)\big] \\
& = -n^2(I-E_{\tau})(x).
\end{align*}
\end{proof}
The above lemma shows that for any Lindblad generator $\mc L$, we have $\Gamma_{\mc L} \le C \Gamma_{-(I-E_{\tau})}$ for some $C>0$, by the comparison principle because the Kraus operators of $-(I-E_{\tau})$ consist of the whole orthonormal basis. If a Lindblad generator with $\sigma$-detailed balance has fixed point algegra $\mc N$, we can show a finer result 
\[
\Gamma_{\mc L} \le C\Gamma_{-(I - E_{\mc N,\tau})}, C>0.\]
To prove the above result, we need some structural analysis of the algebra $\mc N$:
\begin{prop} \label{I-E:1}
There exist a selfadjoint orthonormal set $\{I_n, V_j; j\in \mathcal{J}\}$ which spans $\mc N'$(commutator of $\mc N$), such that
$$\sum_{j \in \mathcal{J}}\big[V_jXV_j - \frac{1}{2}(V_j^2X + XV_j^2) \big]= -( I-E_{\mathcal{N},\tau})(X), X \in \mb M_n.$$
\end{prop}

\begin{proof}
Without loss of generality, we can assume the fixed point algebra $$\mathcal{N} = \bigoplus_{k=1}^m M_{n_k} \otimes 1_{r_k}\subseteq \mathcal{M} =\mb M_n$$
Indeed, from the well-known characterization of finite dimensional $C^*$-algebra, there exists a unitary $U\in \mb M_n$, such that 
\begin{align*}
U^*\mathcal{N}U = \bigoplus_{k=1}^m M_{n_k} \otimes 1_{r_k}
\end{align*}
Then from the uniqueness of conditional expectation determined by states, we have 
\begin{align*}
E_{U^*\mc NU, \tau}(X) = U^*E_{\mc N,\tau}(UXU^*)U, X\in \mb M_n.
\end{align*}
If we can show that $-(I-E_{U^*\mc NU, \tau}) = \sum_{j} \mc L_{V_j}$, where $\{I_n, V_j\}$ is a selfadjoint orthonormal set spanning $\mc (U^*NU)'$, then for any $X \in\mb M_n$,
\begin{align*}
-(I - E_{\mc N,\tau})(X) & = - U(I-E_{U^*\mc NU, \tau}(U^*XU)) U^* \\
& = U\big(\sum_{j \in \mathcal{J}}\big[V_j(U^*XU)V_j - \frac{1}{2}(V_j^2U^*XU + U^*XUV_j^2) \big] \big)U^* \\
& = \sum_{j \in \mc J} \mc L_{\wt V_j},
\end{align*}
where $\wt V_j = UV_j U^*$. Since it is easy to see that $\{I_n, \wt V_j; j\in \mc J\}$ form a selfadjoint orthonormal set which spans $\mc N'$, we only need to show the case when $\mathcal{N}$ has the form $\bigoplus_{k=1}^m M_{n_k} \otimes 1_{r_k}\subseteq \mathcal{M} = M_n$.

\noindent \textbf{Step I:} First suppose $\mathcal{N} = M_{n_1} \otimes 1_{r_1}$. Then  we have $\mc N' = 1_{n_1} \otimes M_{r_1}$. Any orthonormal basis of $\mc N'$ is given by $\{1_{n_1} \otimes V_{i,r_1}\}_{1\le i \le r_1^2}$, where $\{V_{i,r_1}\}_{1 \le i \le r_1^2}\ \text{is\ an\ orthonormal\ basis\ of\ }\mb M_{r_1}$, which means $$\Tr(V_{i,r_1}^*V_{j,r_1}) = r_1 \delta_{i,j}$$ For any $X \in \mb M_n$, it can be written as the block form: $$X = \sum_{1 \le u,v \le n_1} e_{uv} \otimes X_{uv}= \begin{pmatrix}
X_{11} & X_{12} & \cdots & X_{1n_1} \\
X_{21} & X_{22} & \cdots & X_{2n_1} \\
\cdots & \cdots & \cdots & \cdots \\
X_{n_11} & X_{n_12} & \cdots & X_{n_1n_1}
\end{pmatrix}$$
where each $X_{uv} \in \mb M_{r_1}$ and $e_{uv}$ is the matrix unit of $\mb M_{n_1}$. We denote 
\begin{equation}
\mc L_{V_{i,r_1}}(X_{uv}) = V_{i,r_1}^*X_{uv}V_{i,r_1} - \frac{1}{2}(V_{i,r_1}^*V_{i,r_1}X_{uv} + X_{uv} V_{i,r_1}^*V_{i,r_1})
\end{equation} 
Then by calculation, we have 
\begin{align*}
\sum_{i} \mc L_{1_{n_1} \otimes V_{i,r_1}}(X)& = \sum_{1\le u,v \le n_1} e_{uv} \otimes \mc L_{V_{i,r_1}}(X_{uv}) = \sum_{1\le u,v \le n_1} e_{uv} \otimes \big(-r_1^2(X_{uv} - \frac{1}{r_1}\Tr(X_{uv}) 1_{r_1}) \big) \\
& = -r_1^2 \big( \sum_{1\le u,v \le n_1} e_{uv} \otimes X_{uv} -  \sum_{1\le u,v \le n_1} \frac{1}{r_1}\Tr(X_{uv}) e_{uv} \otimes 1_{r_1}\big)\\
& = -r_1^2(X - \frac{1}{r_1}\Tr_{\mc K_1}(X) \otimes 1_{r_1}) = -r_1^2 (I-E_{\mc N,\tau})(X).
\end{align*}
For the second equality, we used Lemma \ref{I-E:1}, and for the last equality, we used the form of conditional expectation when $\mc N = \mb M_{n_1} \otimes 1_{r_j}$, $\mc H = \mc H_1 \otimes \mc K_1$.

\noindent \textbf{Step II}: Suppose $\mathcal{N} = \bigoplus_{k=1}^l \mb M_{n_k} \otimes 1_{r_k}\subseteq \mathcal{M} = \mb M_n$, $l\ge2$. In this case, the decomposition of the underlying Hilbert space is given by
$\mc H = \bigoplus_{k=1}^l \mc H_k \otimes \mc K_k.$
The commutator is given by $\mathcal{N}' = \bigoplus_{k=1}^k 1_{n_k} \otimes M_{r_k}$. Any element $A \in \mathcal{N}'$ has the following block diagonal form 
$$A = diag(1_{n_1}\otimes A_1, \cdots, 1_{n_k}\otimes A_k, \cdots, 1_{n_l}\otimes A_{l}),$$
where $A_{r_k} \in \mb M_{r_k}$. For each $1\le k\le l$, choose a selfadjoint orthonormal basis $\{V_{i,r_k}\}_{1\le i \le r_k^2}$ of $\mb M_{r_k}$, i.e., they are orthonormal and we have 
\begin{equation}
\Tr(V_{i,r_k}^*V_{i,r_k}) = r_k
\end{equation}
Now we define 
\begin{equation}
\begin{aligned}
& \widehat{V_{i,r_1}}: = diag(1_{n_1}\otimes V_{i,r_1},\cdots, 0, \cdots, 0) \\
& \widehat{V_{i,r_k}}: = diag(0,\cdots,0,\cdots, 1_{n_k}\otimes \frac{r_1}{r_k}V_{i,r_k}, \cdots, 0, \cdots, 0),\ 2 \le k \le m
\end{aligned}
\end{equation}
It is clear that $\{\widehat{V_{i,r_j}}\}_{1\le i \le r_k^2, 1 \le k \le l}$ is a selfadjoint orthonormal set whose linear span is $\mathcal{N}'$. We claim that 
\begin{equation}
\sum_{1 \le k \le l}\sum_{1\le i \le r_k^2}\mc L_{\widehat{V_{i,r_k}}}(X) = -r_1^2(I-E_{\mathcal{N},\tau})(X), X \in \mb M_n.
\end{equation}
Note that for any $X \in \mb M_n$, it can be written as the following block form: 
\begin{align*}
X = \begin{pmatrix}
X_{11} & \cdots & X_{1l} \\
\vdots & \vdots & \vdots \\
X_{l1} & \cdots & X_{ll}
\end{pmatrix}
, X_{uv} \in \mb M_{n_ur_u \times n_v r_v},\ \forall u,v = 1,\cdots, l.
\end{align*}
Then we have 
\begin{align*}
& \mc L_{\widehat{V_{i,r_k}}}(X) = \widehat{V_{i,r_j}} X \widehat{V_{i,r_j}} - \frac{1}{2}(\widehat{V_{i,r_j}}^2 X + X\widehat{V_{i,r_j}}^2) \\
& = \frac{r_1^2}{r_k^2} diag(0,\cdots,0, (1_{n_k}\otimes V_{i,r_k})X_{kk}(1_{n_k}\otimes V_{i,r_k}), 0, \cdots, 0) \\
& - \frac{r_1^2}{2r_k^2}\begin{pmatrix}
0 & \cdots & 0\\
\vdots & \vdots & \vdots \\
(1_{n_k}\otimes V_{i,r_k})^2X_{k1} & \cdots & (1_{n_k}\otimes V_{i,r_k})^2X_{kl} \\
\vdots & \vdots & \vdots \\
0 & \cdots & 0
\end{pmatrix} 
 - \frac{r_1^2}{2r_k^2}\begin{pmatrix}
0 & \cdots & X_{1k}(1_{n_k}\otimes V_{i,r_k})^2 & \cdots & 0\\
\cdots & \cdots & \cdots \\
0 & \cdots & X_{kk}(1_{n_k}\otimes V_{i,r_k})^2 & \cdots & 0\\
\cdots & \cdots & \cdots \\
0 & \cdots & X_{lk}(1_{n_k}\otimes V_{i,r_k})^2 & \cdots & 0\\
\end{pmatrix}.
\end{align*}
Then using the calculation in \textbf{Step I}, we have 
\begin{align*}
\sum_{1\le i\le r_k^2} (1_{n_k}\otimes V_{i,r_k})^2 = r_k^2 (1_{n_k}\otimes 1_{r_k}),
\end{align*}
and 
\begin{align*}
\sum_{1\le i\le r_k^2} (1_{n_k}\otimes V_{i,r_k})X_{kk}(1_{n_k}\otimes V_{i,r_k}) = r_k^2 \frac{1}{r_k}\Tr_{\mc K_k}(X_{kk}) \otimes 1_{r_k}.
\end{align*}
Therefore,
\begin{align*}
& \sum_{1\le i\le r_k^2}\mc L_{\widehat{V_{i,r_k}}}(X) = -r_1^2 \begin{pmatrix}
0 & \cdots & \frac{1}{2}X_{1k} & \cdots & 0\\
\cdots & \cdots & \cdots \\
\frac{1}{2}X_{k1} & \cdots & (X_{kk} - \frac{1}{r_k}\Tr_{\mc K_k}(X_{kk}) \otimes 1_{r_k}) & \cdots & \frac{1}{2}X_{kl}\\
\cdots & \cdots & \cdots \\
0 & \cdots & \frac{1}{2}X_{lk} & \cdots & 0\\
\end{pmatrix}
\end{align*}
Then we sum $k$ from $1$ to $l$, we have 
\begin{equation*}
\begin{aligned}
& \sum_{1 \le k \le l}\sum_{1\le i \le r_k^2}\mc L_{\widehat{V_{i,r_k}}}(X) \\
& = -r_1^2 \big(X - diag(\frac{1}{r_1}\Tr_{\mc K_1}(X_{11})\otimes 1_{r_1},\cdots, \frac{1}{r_k}\Tr_{\mc K_k}(X_{kk})\otimes 1_{r_k}, \cdots, \frac{1}{r_l}\Tr_{\mc K_l}(X_{ll})\otimes 1_{r_l}\big) \\
& = -r_1^2 (I-E_{\mathcal{N},\tau})(X).
\end{aligned}
\end{equation*}
\end{proof}

\begin{cor} \label{comparison: upper bound}
Suppose $\mc L$ is a Lindblad generator with $\sigma-$detailed balance, the fixed point algebra of which is given by $\mathcal{N}$. Then there exists a constant $C>0$, such that 
\begin{align*}
\Gamma_{\mc L} \le C \Gamma_{-(I-E_{\mathcal{N},\tau})}.
\end{align*}
\end{cor}
\begin{proof}
We can write down the standard form of $\mc L$:
\begin{align*}
\mc L(x) = \sum_{j=1}^m \big[V_j^*xV_j - \frac{1}{2}(V_j^*V_jx + x V_j^*V_j)\big],
\end{align*}
where $V_j \in \mc N'$. By Proposition \ref{I-E:1}, there exists an orthonormal set $\{I_n, \wt V_j;j \in \mc J\}$ such that 
$$-(I-E_{\mathcal{N},\tau}) = \sum_{j\in \mc J} \big[\wt V_j^*x\wt V_j - \frac{1}{2}(\wt V_j^*\wt V_jx + x\wt V_j^*\wt V_j)\big].$$
Moreover, $\{I_n,\wt V_j;j\in \mc J\}$ spans $\mc N'$. Therefore, we have 
\begin{align*}
span\{V_j: 1\le j \le m\} \subseteq span\{I_n,\wt V_j;j\in \mc J\}.
\end{align*}
Then from our comparison theorem, we know that there exists $C>0$, such that 
$\Gamma_{\mc L} \le C \Gamma_{-(I-E_{\mathcal{N},\tau})}.$
\end{proof}
The following corollary is a simple proof of \cite[Theorem 5.4]{GJL20}:
\begin{cor}
The set of Lindblad generators with $\sigma-$detailed balance which satisfy $\Gamma \mathcal{E}$ is dense with $\|\cdot\|_{2\to 2}$.
\end{cor}
\begin{proof}
Similar to the proof above, for any Lindblad generator $\mc L$ with $\sigma-$detailed balance, the standard form is given by
\begin{align*}
\mc L(x) = \sum_{j=1}^m \big[V_j^*xV_j - \frac{1}{2}(V_j^*V_jx + x V_j^*V_j)\big],
\end{align*}
We can extend $\{V_j: 1\le j \le m\}$ to $\{V_j: 1\le j \le m+l\}$ such that $$span\{V_j: 1\le j \le m+l\} = \mc N'.$$
Then define $\wt{\mc L} = \mc L + \varepsilon \sum_{j=m+1}^{m+l}\mc L_{V_j}, \varepsilon>0$. $\wt{\mc L}$ satisfies $\Gamma_{\wt{\mc L}} \ge C\Gamma_{-(I-E_{\mc N,\tau})}$ for some $C>0$ by comparison theorem. Moreover, it satisfies $\|\mc L - \wt{\mc L}\|_{2\to 2} = \varepsilon \| \sum_{j=m+1}^{m+l}\mc L_{V_j}\|_{2\to 2} \lesssim \varepsilon$. 
\end{proof}

\section{Order relation from norm estimates}
In this section, we obtain order relation of two Lindblad generators $\mc L,\mc L': L^2(\mc M,\tau) \to L^2(\mc M,\tau)$ from the norm $\|\mc L-\mc L' \|_{2\to 2}$. Our main tool is ordered real vector space with an order unit. First we briefly review the basics of general theory of ordered real vector space, which help us connect the order and the norm. For the whole story, we refer the reader to \cite{Paulsen09}
\subsection{Ordered real vector space}
\begin{definition}
Suppose $V$ is a real vector space. $\emptyset \neq C \subseteq V$ is called a \textit{cone} if 
\begin{itemize}
\item $a\cdot v \in C$ whenever $a \in [0,\infty)$ and $v \in C$.
\item $v + w \in C$ whenever $v,w\in C$.
\end{itemize}
An ordered vector space is a pair $(V,V^+)$ consisting of a real vector space $V$ and a cone $V^+\subseteq V$ satisfying $V^+ \cap -V^+ = \{0\}$. We say that $V^+$ is full if $V= V^+-V^+$.
\end{definition}
We can naturally define a partial ordering $\le$ on $V$ by 
\begin{equation}
v \le w \iff w-v \in V^+.
\end{equation}
In particular, $v \ge 0 $ if and only if $ v \in V^+$. For this reason, we call $V^+$ the \textit{the cone of positive elements} of $V$.
\begin{definition}
Suppose $(V,V^+)$ is an ordered vector space. An element $e$ is called an order unit for $V$, if for any $v \in V$, there exists a constant $C=C(v)>0$, such that $v \le Ce$. For any $v \in V$, define 
\begin{equation}\label{order norm definition}
||v||_{or} = \inf \{r \ge 0: -re \le v \le re\}
\end{equation}
\end{definition}
The following property is crucial, which is proved in \cite[Proposition 2.23]{Paulsen09}:
\begin{prop}\label{ordered norm}
$||\cdot||_{or}$ gives a semi-norm on $V$. Moreover, if $e$ is Archimedean, i.e., when $v\in V$ with $re + v \ge 0$ for all $r>0$, then $v \in V^+$, the semi-norm is actually a norm. 
\end{prop}
The above property enables us to connect the general vector space $V$ to our setting. Recall that if $L$ is a Lindblad generator, we can define its gradient matrix $m_L$ by \eqref{big matrix}. For two Lindblad generators $L,L'$, we have $L \le L'$ if and only if $m_{L} \le m_{L'}$. The first inequality is defined by its gradient form and the second inequality is the usual matrix order.
\begin{example} \label{ex:1}
We define an ordered vector space with full positive cone: 
\begin{equation}
V^+:= \{m_{L} \in \mb M_{n^2}(\mb M_n): L\ \text{is\ a\ Lindblad\ generator}: L^2(\mc M,\tau) \to L^2(\mc M,\tau)\}.
\end{equation}
Define $e = m_{-(I-E_{\tau})}$. $e$ is an order unit by Lemma \ref{I-E:primitive}. Moreover, $e$ is an Archimedean order unit. 
\end{example}
To see why $e$ is Archimedean, suppose $v \in V$, then there exists two Lindblad generators $L, L'$ such that $v = m_{L} - m_{L'}$. By Lemma \ref{I-E:primitive}, we can find an orthonormal set $\{V_j\}_{j \in \mc J}$ such that $\{I_n, V_j;j\in \mc J\}$ spans $\mc M = \mb M_n$ and 
$$-(I-E_{\tau})(x) = \sum_{j \in \mc J} \big[V_j^*xV_j - \frac{1}{2}(V_j^*V_jx + xV_j^*V_j)\big].$$
By Theorem \ref{Main theorem} (3), we can write $L,L'$ as 
\begin{align*}
Lx = i[H_1,x]+ \sum_{i,j\in \mc J} c_{ij}^1 \big[V_i^*xV_j- \frac{1}{2}(V_i^*V_jx + x V_i^*V_j)\big],\\
L'x = i[H_2,x]+ \sum_{i,j\in \mc J} c_{ij}^2 \big[V_i^*xV_j- \frac{1}{2}(V_i^*V_jx + x V_i^*V_j)\big],
\end{align*}
where $(c_{ij}^1),(c_{ij}^2)$ are positive matrices. To prove $e$ is Archimedean, suppose $v + re \ge 0$ for all $r>0$, we have $m_{L-L'-r(I-E_{\tau})} \ge 0$, which means
$$L-L'-r(I-E_{\tau}) = i[H_1-H_2,x] + \sum_{i,j\in \mc J} (c_{ij}^1 - c_{i,j}^2 + r\delta_{ij}) \big[V_i^*xV_j- \frac{1}{2}(V_i^*V_jx + x V_i^*V_j)\big]$$
is a Lindblad generator. Applying Theorem \ref{Main theorem} (3) again, we know that $(c_{ij}^1)-(c_{ij}^2)+r I_{|\mc J|}$ is a positive matrix for all $r >0$. Therefore, $(c_{ij}^1)-(c_{ij}^2)$ is a positive matrix, which implies $L-L'$ is a Lindblad generator thus $v = m_{L-L'} \in V^+$. 

Next we consider another example. Fix a Lindblad generator $L$ with $\sigma-$detailed balance, the fixed point algebra of which is $\mc N$. By Corollary \ref{comparison: upper bound}, we have \begin{equation} \label{upper gamma E}
\Gamma_L \le C \Gamma_{-(I- E_{\mc N},\tau)},C>0 \iff m_{L} \le C e
\end{equation} 
where $e = m_{-(I- E_{\mc N,\tau})}$. 
\begin{example}\label{ex:2}
Define an ordered vector space $(V,V^+,e)$ where $V = V^+ - V^+$ and 
\begin{equation} \label{ordered space}
V^+ = \{m_{\wt L}: \wt L \text{\ is\ a\ Lindblad\ generator\ with}\ \sigma-\text{detailed\ balance},\ m_{\wt L} \le c(\wt L)e,\exists c(\wt L)>0\}.
\end{equation}
\end{example}
Using the same argument as the previous example and Proposition \ref{I-E:1}, the order unit is Archimedean.

\subsection{Order relation from norm estimates}
Given a general Lindblad generator $L$, if another Lindblad generator $L'$ is close to $L$ under $\|\cdot\|_{2\to 2}$ norm, we aim to show some order relation if $\|L-L'\|_{2\to 2}$ is small enough. 
\begin{prop}\label{stability of gradient form}
Suppose $L$ is a Lindblad generator. Then $\forall \varepsilon>0$, there exists $\delta = \delta(\varepsilon,L)>0$, such that for any Lindblad generator $L'$,
\begin{equation}
||L-L'||_{2\to 2} < \delta \implies
(1-\varepsilon)\Gamma_L \le \Gamma_{L'} \le \Gamma_L + \varepsilon \Gamma_{-(I-E_{\tau})}.
\end{equation}
\end{prop}
\begin{proof}
We use the constructed ordered vector space in Example \ref{ex:1}. Since the order unit $e= m_{-(I-E_{\tau})}$ is Archimedean, by Proposition \ref{ordered norm}, the order norm $\|\cdot\|_{or}$ is actually a norm. Moreover, since $\dim V <\infty$, any two norms are equivalent, we have
$$\frac{1}{c_1} ||v|| \le ||v||_{or} \le c_1 ||v||,\ v\in V,\ c_1\ge1.$$
Suppose $||L-L'||<\delta$, where $\delta$ is to be determined. Then by definition of $$m_{L} = \big(L(e_{rs}^*e_{tv}) - L(e_{rs}^*)e_{tv} - e_{rs}^*L(e_{tv})\big)_{rs,tv},$$ see \eqref{big matrix}, we can find $c_2>0$, such that 
$$\|m_{L} -m_{L'}\| = \|\big((L-L')(e_{rs}^*e_{tv}) -(L-L')(e_{rs}^*)e_{tv} - e_{rs}^*(L-L')(e_{tv})\big)_{rs,tv}\| < c_2 \delta $$
Then by equivalence of the norms, we have 
$$||m_L-m_{L'}||_{or} < c_1c_2 \delta$$
Then by definition of $||\cdot||_{or}$ \eqref{order norm definition}, we have 
\begin{equation}\label{key inequality}
m_L - c_1c_2\delta e \le m_{L'} \le m_L + c_1c_2\delta e.
\end{equation}
Suppose 
$$Lx = i[H,x]+ \sum_{j=1}^m \big[V_j^*xV_j - \frac{1}{2}(V_j^*V_jx + xV_j^*V_j)\big]$$
is the standard form of Lindblad generator. We also write down the general form of $-(I-E_{\tau})$:
\begin{align*}
-(I-E_{\tau})(x)= \sum_{j=1}^m \big[F_j^*xF_j - \frac{1}{2}(F_j^*F_jx + xF_j^*F_j)\big] + \sum_{j=m+1}^{n^2-1} \big[F_j^*xF_j - \frac{1}{2}(F_j^*F_jx + xF_j^*F_j)\big],
\end{align*} 
where we assume $span\{F_j:1\le j \le m\} = span\{V_j:1\le j \le m\}$.

By \eqref{key inequality}, we have by linearity 
\begin{align*}
m_{L'} \ge m_{L_1} - c_1c_2\delta m_{L_2},
\end{align*}
where 
\begin{align*}
L_1 = \sum_{j=1}^{m} \mc L_{V_j} - c_1c_2\delta \sum_{j=1}^m \mc L_{F_j},\ L_2 = \sum_{j = m+1}^{n^2-1} \mc L_{F_j}.
\end{align*}
Since the jump operators of $L_1$ and $L_2$ are orthogonal, by Corollary \ref{drop}, we have $m_{L'} \ge m_{L_1}$.
Using the fact $span\{F_j:1\le j \le m\} = span\{V_j:1\le j \le m\}$, there exists a constant $c(L)>0$, such that
$$c(L) m_{\sum_{j=1}^{m} \mc L_{V_j}} \ge m_{\sum_{j=1}^m \mc L_{F_j}}.$$ 
Choosing $0<\delta < \frac{\varepsilon}{c(L)c_1c_2}$, we have 
$$m_{L'} \ge m_{L_1} \ge (1 - c(L)c_1c_2\delta)m_{L} \ge (1-\varepsilon)m_{L}.$$
Then applying Proposition \ref{comparison}, we finish the proof of the first inequality. The other inequality is obvious if we choose $0<\delta < \frac{\varepsilon}{c_1c_2}$.
\end{proof}
If we restrict ourselves to the class of Lindblad generators with $\sigma-$detailed balance, using the same argument as above and Example \ref{ex:2}, we have the following order relation: 
\begin{prop}\label{stability of gradient form: symmetric}
Suppose $L$ is a Lindblad generator with $\sigma-$detailed balance. $\mc N$ is its fixed point algebra. Then $\forall \varepsilon>0$, there exists $\delta = \delta(\varepsilon,L)> 0$, such that for any Lindblad generator $L'$ with $\sigma-$detailed balance with the same fixed point algebra $\mc N$, we have
\begin{equation}
\|L-L'\|_{L_2 \to L_2} < \delta \implies
(1-\varepsilon)\Gamma_{L} \le \Gamma_{L'} \le \Gamma_{L} + \varepsilon \Gamma_{-(I-E_{\mc N,\tau})}.
\end{equation}
\end{prop}

\section{Application: Stability property of non-commutative functional inequalites}
In this section, we apply the comparison principle of quantum Markov semigroups and the results established in Section 4,5 to show stability properties for two functional inequalities, \textit{Poincar\'e inequality} and \textit{complete modified Logarithmic Sobolev inequality}. 
\subsection{Stability of Poincar\'e inequality}
Suppose $L: \mc M \to \mc M$ is a Lindblad generator with $\sigma-$detailed balance. Its fixed point algebra is $\mc N$. We say that the quantum Markov semigroup $T_t = e^{tL}$ satisfies \textit{Poincar\'e inequality} for some $\lambda>0$ if 
\begin{equation}\label{Poincare}
\text{Var}_{\mc N,\sigma}(T_t(X)) \le e^{-\lambda t} \text{Var}_{\mc N,\sigma}(X),\forall X \in \mc M.
\end{equation}
The variance $\text{Var}_{\mc N,\sigma}$ is defined via the BKM-inner product.
\begin{equation}
\text{Var}_{\mc N,\sigma}(X) = \int_0^1 \Tr\big((X-E_{\mc N,\sigma}(X))^*\sigma^s(X-E_{\mc N,\sigma}(X))\sigma^{1-s}\big)ds. 
\end{equation}
By taking the derivative of $\text{Var}_{\mc N,\sigma}(T_t(X))$, it is easy to see the inequality \eqref{Poincare} is equivalent to 
\begin{equation}
\lambda \text{Var}_{\mc N,\sigma}(X) \le \mc E_{L}(X).
\end{equation}
We call $\mc E_{L}$ the \textit{Dirichlet form} associated to $L$, which is defined as 
\begin{equation}
\mc E_{L}(X,Y):= \int_0^1 \Tr(X^* \sigma^s L(Y)\sigma^{1-s})ds
\end{equation}
and we denote $\mc E_{L}(X,X):= \mc E_{L}(X)$. We call the largest constant $\lambda(L)$ such that \eqref{Poincare} holds the \textit{spectral gap} of $L$, which is given by 
\begin{equation}
\lambda(L):= \inf_X \frac{\mc E_{L}(X)}{\text{Var}_{N,\sigma}(X)}.
\end{equation}
Then our stability result for \textit{Poincar\'e inequality} is given as follows:
\begin{theorem}\label{Stability of PI}
Suppose $L: \mc M \to \mc M$ is a Lindblad generator with $\sigma-$detailed balance. Then for any $\varepsilon>0$, there exists $\delta = \delta(\varepsilon, \mc L)>0$ such that for any Lindblad generator $L'$ with $\sigma$-detailed balance with the same fixed point algebra as $L$, we have 
\begin{align*}
\|L-L'\|_{L_2 \to L_2} < \delta \implies
(1-\varepsilon) \lambda(L) \le \lambda(L').
\end{align*}
\end{theorem}
\begin{proof}
Under the assumption above, using Proposition \ref{stability of gradient form: symmetric}, we know that $(1-\varepsilon)\Gamma_{L} \le \Gamma_{L'}$. We claim that we have 
\begin{equation}
(1-\varepsilon)\mc E_{L}(X) \le \mc E_{L'}(X), \forall X \in \mc M.
\end{equation}
Recall that $L = \sum_{j \in \mc J} e^{-\omega_j/2}\mc L_{V_j}$. We rewrite the expression of $L,L'$ as in \eqref{generator partition}: suppose $\mc J = \{\mc J_1,\cdots,\mc J_l\}$, such that for each $1\le k\le l$ and $j\in \mc J_k, \omega_j = \omega_k$
\begin{equation}\label{partition}
L = \sum_{k=1}^l e^{-\omega_k/2}\sum_{j \in \mc J_k} \mc L_{V_j},
\end{equation}
where $\{V_j\}_{j \in \mc J_k}$ are orthonormal eigenvectors of $\Delta_{\sigma}$ corresponding to the eigenvalue $e^{-\omega_k/2}$. Since $\Gamma_L \le C \Gamma_{L'}$ and $L'$ satisfies $\sigma$-detailed balance condition. We can write $L'$ as 
\begin{align*}
L' = \sum_{k=1}^l e^{-\omega_k/2}\sum_{j \in \mc J_k} \mc L_{\wt V_j} + \text{remainder},
\end{align*}
where the remainder term is orthogonal to $L$ and $span\{V_j\}_{j \in \mc J} = span\{\wt V_j\}_{j \in \mc J}$. The corresponding matrix is given by $A \in \mb M_m$, $|\mc J|= m$:
$$(V_j)_{j \in \mc J} = (A\otimes I_n)(\wt V_j)_{j \in \mc J}.$$ 
Moreover, for each $1\le k \le l$ we have $A_k \in \mb M_{m_k}, |\mc J_k| = m_k$ such that
\begin{align*}
(V_j)_{j \in \mc J_k}
= (A_k\otimes I_n) (\wt V_j)_{j \in \mc J_k}
\end{align*}
and $A = diag(A_1,\cdots, A_k,\cdots,A_l)$. Define $\Omega_{\delta_{L_k}} = \Omega_{\delta_{\sum_{j\in \mc J_k}\mc L_{V_j}}}$ and $\Omega_{\delta_{L_k'}} = \Omega_{\delta_{\sum_{j\in \mc J_k}\mc L_{\wt V_j}}}$as in \eqref{defn:bimodule}. Then the module map defined by \eqref{defn:module} 
$
T_k: \Omega_{\delta_{L_k'}} \to \Omega_{\delta_{L_k}}$ is given by $A_k \otimes I_n$. By direct calculation, the Dirichlet form is given by \cite{CM17}
\begin{align*}
\mc E_L(X) & = \int_0^1 \sum_{j\in \mc J}e^{(\frac{1}{2}-s)\omega_j}\Tr\big([V_j,X]^*\sigma^2 [V_j,X]\sigma^{1-s}\big)ds \\
& = \sum_{k=1}^l \int_0^1 e^{(\frac{1}{2}-s)\omega_k}\sum_{j\in \mc J_k}\Tr\big((\sigma^{\frac{s}{2}}[V_j,X]\sigma^{\frac{1-s}{2}})^*(\sigma^{\frac{s}{2}}[V_j,X]\sigma^{\frac{1-s}{2}})\big)ds \\
& \le \sum_{k=1}^l C_k \int_0^1 e^{(\frac{1}{2}-s)\omega_k}\sum_{j\in \mc J_k}\Tr\big((\sigma^{\frac{s}{2}}[\wt V_j,X]\sigma^{\frac{1-s}{2}})^*(\sigma^{\frac{s}{2}}[\wt V_j,X]\sigma^{\frac{1-s}{2}})\big)ds \\
& \le \frac{1}{1-\varepsilon}\sum_{k=1}^l \int_0^1 e^{(\frac{1}{2}-s)\omega_k}\sum_{j\in \mc J_k}\Tr\big((\sigma^{\frac{s}{2}}[\wt V_j,X]\sigma^{\frac{1-s}{2}})^*(\sigma^{\frac{s}{2}}[\wt V_j,X]\sigma^{\frac{1-s}{2}})\big)ds \le  \frac{1}{1-\varepsilon} \mc E_{L'}(X).
\end{align*}
For the first inequality above, we used \eqref{module map}:
\begin{align*}
\Tr \langle T_k(\xi),T_k(\xi) \rangle_{\Omega_{\delta_{L_k}}} \le C_k \Tr\langle \xi, \xi \rangle_{\Omega_{\delta_{L_k'}}},
\end{align*}
for $C_k = \|A_k^*A_k\|$ and $\xi = \sigma^{s/2}\delta_{L_k'}(X) \sigma^{(1-s)/2}$. For the second inequality, we used the fact that $C_k = \|A_k^*A_k\| \le \|A^*A\|\le \frac{1}{1-\varepsilon}$. Thus it is easy to see that 
\begin{align*}
\lambda(L) = \inf_X \frac{\mc E_{L}(X)}{\text{Var}_{\mc N,\sigma}(X)} \le \frac{1}{1-\varepsilon} \inf_X \frac{\mc E_{L'}(X)}{\text{Var}_{\mc N,\sigma}(X)} = \frac{1}{1-\varepsilon}\lambda(L').
\end{align*}
\end{proof} 

\subsection{Stability of complete modified Logarithmic Sobolev inequality}
Suppose $L: \mc M \to \mc M$ is a Lindblad generator with $\sigma-$detailed balance. $T_t = e^{tL}$ is the quantum Markov semigroup generated by $L$. Then we say that $L$ satisfies $\alpha-$\textit{modified Logarithmic Sobolev inequality} for some $\alpha>0$, if for any $\rho \in \mc D^+(\mc H)$
\begin{equation}\label{MLSI}
D(T_{t*}(\rho) \| \mc E_{\mc N *}(\rho)) \le e^{-\alpha t} D(\rho \| \mc (E_{\mc N})_*(\rho)),
\end{equation}
where $D(\cdot\|\cdot)$ is the relative entropy and $E_{\mc N}$ is the long time limit of $T_t$: $\lim_{t \to \infty} T_t = E_{\mc N}$ \cite{FM82}. A stronger version called \textit{complete modified Logarithmic Sobolev inequality} is introduced in \cite{GJL20}. We say $L$ satisfies $\alpha-$\textit{complete modified Logarithmic Sobolev inequality} for some $\wt \alpha>0$ if for any finite dimensional reference system $R$ and $\rho \in \mc D(\mc H \otimes \mc H_R)$,
\begin{equation}\label{complete MLSI}
D((T_{t*}\otimes id_R)(\rho) \| (\mc E_{\mc N *}\otimes id_R)(\rho)) \le e^{-\wt \alpha t} D(\rho \| (\mc E_{\mc N *}\otimes id_R)(\rho)).
\end{equation}
By taking the derivative, the equation \eqref{complete MLSI} is equivalent to the following inequality
\begin{equation}\label{CMLSI}
\wt \alpha D(\rho \|(\mc E_{\mc N *}\otimes id_R)(\rho) ) \le \text{EP}_{L\otimes id_R}(\rho)
\end{equation}
for any finite dimensional reference system $R$ and $\rho \in \mc D(\mc H \otimes \mc H_R)$. $\text{EP}_{L\otimes id_R}(\rho)$ is defined by 
\begin{align*}
\text{EP}_{L\otimes id_R}(\rho): = \frac{d}{dt}\bigg|_{t=0}D((T_{t*}\otimes id_R)(\rho)\|(\mc E_{\mc N *}\otimes id_R)(\rho)).
\end{align*}
Suppose $Lx = \sum_{j \in \mc J} e^{-\omega_j /2} \big[V_j^*xV_j - \frac{1}{2}(V_j^*V_jx + xV_j^*V_j)\big]$, we define $\partial_j(X): = [V_j,X], j \in \mc J$, and $\Gamma_{\sigma,\frac{1}{2}}(X):= \sigma^{1/2}X \sigma^{1/2}$ for any $X \in \mc M$. It is shown in \cite{JLR19} that 
$$\text{EP}_{L}(\rho) = \sum_{j\in \mc J} \big\langle \Gamma_{\sigma,\frac{1}{2}}\circ \partial_j \circ \Gamma_{\sigma,\frac{1}{2}}^{-1}(\rho), [\rho]_{\omega_j}^{-1}(\Gamma_{\sigma,\frac{1}{2}}\circ \partial_j \circ \Gamma_{\sigma,\frac{1}{2}}^{-1}(\rho)) \big\rangle_{HS},$$
where 
\begin{equation}
[\rho]_{\omega_j}^{-1}(X):= \int_0^{\infty}(r+ e^{-\omega_j/2}\rho)^{-1}X(r+ e^{\omega_j/2}\rho)^{-1}dr.
\end{equation}
Now we denote $\text{CMLSI}(L)$ as the optimal constant $\wt \alpha>0$ such that \eqref{CMLSI} holds, i.e., 
\begin{equation}
\text{CMLSI}(L) = \inf_R\inf_{\rho \in \mc D(\mc H\otimes \mc H_R)}\frac{\text{EP}_{L\otimes id_R}(\rho)}{D(\rho \|(\mc E_{\mc N *}\otimes id_R)(\rho) )}
\end{equation}

It is shown in \cite{GJL21, GR21} that $\text{CMLSI}(L)>0$ for any $L$ with $\sigma-$detailed balance. Applying our comparison theorem, we can show the following stability property of complete Logarithmic Sobolev inequality:
\begin{theorem}\label{Stability of CLSI}
Suppose $L: \mc M \to \mc M$ is a Lindblad generator with $\sigma-$detailed balance. Its fixed point algebra is $\mc N$. Then $\forall \varepsilon>0$, there exists $\delta = \delta(\varepsilon,L)> 0$, such that for any Lindblad generator $L'$ with $\sigma-$detailed balance which has fixed point algebra $\mc N$, we have
\begin{equation}
\|L-L'\|_{L_2 \to L_2} < \delta \implies
(1-\varepsilon)\text{CMLSI}(L)\le \text{CMLSI}(L').
\end{equation}
\end{theorem}
\begin{proof}
Similar to the proof of Theorem \ref{Stability of PI}, using Proposition \ref{stability of gradient form: symmetric}, we know that $(1-\varepsilon)\Gamma_{L} \le \Gamma_{L'}$. We claim that we have 
\begin{equation}
(1-\varepsilon) \text{EP}_{L}(\rho) \le \text{EP}_{L'}(\rho), \forall \rho \in \mc D(\mc H).
\end{equation}
Using the same trick \eqref{partition} as before, we can rewrite $L,L'$ as in the proof of Theorem \ref{Stability of PI}. Denote $\sigma^{1/2}X\sigma^{1/2} = \rho$, $\text{EP}_L(\rho)$ can be written as 
\begin{align*}
\text{EP}_L(\rho) = \sum_{k=1}^l \sum_{j \in \mc J_k} \Tr((\sigma^{1/2}[V_j,X]\sigma^{1/2})^{*} [\rho]_{\omega_k}^{-1}(\sigma^{1/2}[V_j,X]\sigma^{1/2})).
\end{align*}
Now we use the elementary integral 
\begin{align*}
\int_0^{\infty} (r+\lambda)^{-1} (r+\mu)^{-1}dr = \frac{\log \lambda - \log \mu}{\lambda - \mu}, \lambda, \mu \in \mb R^{+}
\end{align*}
and functional calculus, we know that
\begin{align*}
[\rho]_{\omega_k}^{-1} & = \int_{0}^{\infty}(r+e^{-\omega_k/2}L_{\rho})^{-1}(r+e^{\omega_k/2}R_{\rho})^{-1} dr \\
& = \sum_{u,v =1}^n \frac{\log(e^{-\omega_k/2} \lambda_u) -\log(e^{\omega_k/2} \lambda_v) }{e^{-\omega_k/2} \lambda_u - e^{\omega_k/2} \lambda_v}L_{E_u}R_{E_{v}},
\end{align*}
where $L_{\rho}(X):= \rho X, R_{\rho}(X):=X\rho$ and $\rho = \sum_{u=1}^n \lambda_u E_u$ is the spectral decomposition, with $E_u$ being projections.
Then similar to the proof of Theorem \ref{Stability of PI},
\begin{align*}
& EP_{L}(\rho) = \sum_{u,v =1}^n \sum_{k=1}^l \frac{\log(e^{-\omega_k/2} \lambda_u) -\log(e^{\omega_k/2} \lambda_v)}{e^{-\omega_k/2} \lambda_u - e^{\omega_k/2}} \\
& \cdot \sum_{j \in \mc J_k} \Tr\big((E_u\sigma^{1/2}[V_j,X]\sigma^{1/2}E_v)^{*} (E_u\sigma^{1/2}[V_{j},X]\sigma^{1/2}E_v)\big) \\
& \le \sum_{u,v =1}^n \sum_{k=1}^l \frac{\log(e^{-\omega_k/2} \lambda_u) -\log(e^{\omega_k/2} \lambda_v)}{e^{-\omega_k/2} \lambda_u - e^{\omega_k/2}} \\
& \cdot C_k\sum_{j \in \mc J_k} \Tr\big((E_u\sigma^{1/2}[\wt V_j,X]\sigma^{1/2}E_v)^{*} (E_u\sigma^{1/2}[\wt V_{j},X]\sigma^{1/2}E_v)\big) \\
& \le \frac{1}{1-\varepsilon} EP_{L'}(\rho).
\end{align*}
The above proof can be adapted easily so that we have
$$(1-\varepsilon) \text{EP}_{L\otimes id_R}(\rho) \le \text{EP}_{L'\otimes id_R}(\rho), \forall \rho \in \mc D(\mc H \otimes \mc H_R).$$
Then 
\begin{align*}
\text{CMLSI}(L) & = \inf_R\inf_{\rho \in \mc D(\mc H\otimes \mc H_R)}\frac{\text{EP}_{L\otimes id_R}(\rho)}{D(\rho \|(\mc E_{\mc N *}\otimes id_R)(\rho) )} \\
& \le \frac{1}{1-\varepsilon}\inf_R\inf_{\rho \in \mc D(\mc H\otimes \mc H_R)}\frac{\text{EP}_{L'\otimes id_R}(\rho)}{D(\rho \|(\mc E_{\mc N *}\otimes id_R)(\rho) )} = \frac{1}{1-\varepsilon}\text{CMLSI}(L'). 
\end{align*}

\end{proof}
\begin{remark}
Note that our stability theorem is different from the following, which is proved in \cite{JLR19}:
\textit{
Suppose $L$(resp. $L'$) is a generator of quantum Markov semigroup with $\sigma-$(resp. $\sigma'-$) detailed balance condition, with the following form
\begin{equation}
\begin{aligned}
& Lx = \sum_{i \in \mathcal{I}} c_i \big[V_i^*xV_i - \frac{1}{2}(V_i^*V_ix + xV_i^*V_i )\big] \\
& L'x = \sum_{i \in \mathcal{I}} c_i' \big[V_i^*xV_i - \frac{1}{2}(V_i^*V_ix + xV_i^*V_i )\big]
\end{aligned}
\end{equation}
where $c_i,c_i'>0$ for all $i \in \mathcal{I}$, and $\{V_i\}_{i \in \mathcal{I}}$ are eigenvectors of both $\Delta_{\sigma}$ and $\Delta_{\sigma'}$, such that 
$\{V_i\}_{i \in \mathcal{I}} = \{V_i^*\}_{i \in \mathcal{I}}$, $\tau(V_i) = 0$ and $ \tau(V_i^*V_j) = \delta_{ij}$. \\
Then 
\begin{equation}
\min_{k,l}\frac{\sigma_k}{\sigma_l} \min_{i \in \mathcal{I}}\frac{c_i}{c_i'} \text{CLSI}(L') \le \text{CLSI(L)} \le \text{CLSI}(L') \max_{k,l}\frac{\sigma_k}{\sigma_l} \max_{i \in \mathcal{I}}\frac{c_i}{c_i'}
\end{equation}
}
In fact, the assumption of the above theorem requires the two generators must have the standard form given by the same $V_j$. However, by a perturbation, the standard form of two generators may not be given by the same $V_j$. 
\end{remark}

\section{Application: Stability property of $g^{(2)}$ in quantum optics}
Starting from a physical system given by the Green's function of an electric or magnetic field, and the corresponding quantum field, the Lindblad generator is obtained by a number of approximation steps, in particular the Born-Markov approximation. This approximation gives us the evolution of density operators, which is also known as the \textit{master equation}:
\begin{align*}
\frac{d}{dt}\rho(t) = L_*(\rho(t)),\ \rho(0)=\rho.
\end{align*}
where dual map $L_*$ is the dual map
\begin{equation}
L_*(\rho) = i[H,\rho] + \sum_{j} \big[V_j\rho V_j^* - \frac{1}{2}(V_j^*V_j\rho + \rho V_j^*V_j)\big].
\end{equation}
In the setting of quantum optics, the above equation models the evolution of atomic density matrix $\rho(0) = \rho$. An important quantity named \textit{second order correlation function} is define by
\begin{equation}
g^{(2)}(0) = \frac{\Tr(\Psi_*\Psi_*(\rho))}{\Tr(\Psi_*(\rho))^2} = \frac{\sum_{i,j} \langle V_j^*V_i^*V_iV_j \rangle_{\rho}}{(\sum_{j} \langle V_j^*V_j \rangle_{\rho})^2}
\end{equation}
The interpretation of $g^{(2)}(0)$ is the probability of finding two photoelectrons without a delay, compared to the squared probability of finding one photoelectron. If $g^{(2)}(0)>1$, a phenomenon called \textit{superradience} happens, see \cite{AM22}.
We now show a stability property for the quantity $g^{(2)}(0)$. We denote $g^{(2)}_L(0)$ to indicate the relation to the underlying Lindblad generator. We assume implicitly that the initial state is given by the same density $\rho$.
\begin{prop}
Suppose $L$ is a Lindblad generator with the associated jump map $\Psi_L$ and $\rho$ is a density operator. We assume that $g_L^{(2)}(0)>0$. Then for any $\varepsilon>0$, there exists a constant $\delta =\delta(\varepsilon,L,\rho)>0$, such that for any other Lindblad generator $L'$,
\begin{equation}
\|L-L'\|_{2\to 2} \le \delta \implies (1-\varepsilon)g_L^{(2)}(0) \le g_{L'}^{(2)}(0) \le (1+\varepsilon)g_L^{(2)}(0)
\end{equation}
\end{prop}
\begin{proof}
By Proposition \ref{stability of gradient form}, there exists a small constant $\delta'>0$, to be determined, such that 
\begin{align*}
(1-\delta')\Gamma_{L} \le \Gamma_{L'} \le \Gamma_{L} + \delta' \Gamma_{-(I-E_{\tau})}.
\end{align*}
Then apply Corollary \ref{equivalent comparison}, we know that 
\begin{align*}
(1-\delta')\Psi_L \le \Psi_{L'} \le \Psi_L + \delta'\Psi_e = \Psi_L + \delta'E_{\tau}
\end{align*}
where $\Psi_e$ is the jump map of $-(I-E_{\tau})$, which is actually $E_{\tau}$. 
Then we calculate the numerator and denominator of $g^{(2)}(0)$ separately. 
\begin{align*}
& \Tr((\Psi_{L'})_*(\Psi_{L'})_*(\rho)) = \Tr(\rho \Psi_{L'}(\Psi_{L'}(I_n))) \\
& \le \Tr(\rho \Psi_{L}(\Psi_{L}(I_n))) + \delta' \Tr(\rho \Psi_{L}(E_{\tau}(I_n))) + \delta' \Tr(\rho E_{\tau}(\Psi_{L}(I_n))) + \delta'^2\Tr(\rho E_{\tau}(E_{\tau}(I_n))) \\
& = \Tr(\rho \Psi_{L}(\Psi_{L}(I_n))) + \delta'(1+\frac{1}{n})\Tr(\rho \Psi_{L}(I_n)) + \delta'^2
\end{align*}
For the denominator, we have 
\begin{align*}
& \Tr((\Psi_{L'})_*(\rho))^2 = \Tr(\rho \Psi_{L'}(I_n))^2 \\
& \le \big(\Tr(\rho \Psi_{L}(I_n)) + \delta' \Tr(\rho E_{\tau}(I_n))\big)^2 \\
& = \Tr(\rho \Psi_{L}(I_n))^2+ 2\delta' \Tr(\rho \Psi_{L}(I_n)) + \delta'^2.
\end{align*}
On the other hand, it is obvious that 
\begin{align*}
\Tr((\Psi_{L'})_*(\Psi_{L'})_*(\rho)) \ge (1-\delta')^2 \Tr((\Psi_{L})_*(\Psi_{L})_*(\rho)),\ \Tr((\Psi_{L'})_*(\rho))^2 \ge (1-\delta')^2 \Tr((\Psi_{L})_*(\rho))^2.
\end{align*}
Therefore, 
\begin{align*}
\frac{\Tr((\Psi_{L})_*(\Psi_{L})_*(\rho))}{\Tr(\rho \Psi_{L}(I_n))^2+ 2\delta' \Tr(\rho \Psi_{L}(I_n)) + \delta'^2} \le g^{(2)}_{L'}(0) \le \frac{\Tr(\rho \Psi_{L}(\Psi_{L}(I_n))) + \delta'(1+\frac{1}{n})\Tr(\rho \Psi_{L}(I_n)) + \delta'^2}{(1-\delta')^2 \Tr((\Psi_{L})_*(\rho))^2}.
\end{align*}
If we choose $\delta'>0$ such that 
\begin{align*}
2\delta' \Tr(\rho \Psi_{L}(I_n)) + \delta'^2 \le c\varepsilon \min\{\Tr(\rho \Psi_{L}(I_n))^2,\ \Tr(\rho \Psi_{L}(\Psi_{L}(I_n)))\}, c>0,
\end{align*}
we have \begin{align*}
\frac{1}{1+c\varepsilon} g^{(2)}(0) \le  g^{(2)}_{L'}(0) \le \frac{1+c\varepsilon}{(1-\delta')^2} g^{(2)}_{L}(0).
\end{align*}
Then choosing $c>0$ small enough, we arrive at our result.
\end{proof}

\bibliographystyle{plain} 
\bibliography{stabilityref} 
\end{document}